\documentclass[10pt,a4paper]{amsart}

\usepackage{amssymb,amsthm,amsmath,amscd}     	
\usepackage[english]{babel}
\usepackage{mathtools}
\usepackage{booktabs}
\usepackage{capt-of}
\usepackage{siunitx}
\usepackage{float}
\usepackage{pgfplots}
\usepackage{array}
\usepackage{verbatim}
\usepackage{enumitem}
\usepackage{hyperref}

\allowdisplaybreaks

\pgfplotsset{compat=1.10}
\usepgfplotslibrary{fillbetween}
\usetikzlibrary{patterns}

\newtheorem{theorem}{Theorem}[section]
\newtheorem{lemma}[theorem]{Lemma} 
\newtheorem{corollary}[theorem]{Corollary}
\newtheorem{innercustomthm}{Summary of Theorem}

\newenvironment{customthm}[1]
  {\renewcommand\theinnercustomthm{#1}\innercustomthm}
  {\endinnercustomthm}
  
\theoremstyle{remark}
\newtheorem{remark}[theorem]{Remark}

\DeclarePairedDelimiter{\ceil}{\lceil}{\rceil}
\DeclarePairedDelimiter{\floor}{\lfloor}{\rfloor}

\title{On $\tau$-Li-type criterion and explicit zero-free regions in the critical strip}

\date{}

\begin{document}

\title{Explicit zero-free regions and a $\tau$-Li-type criterion}

\author{Neea Paloj\"arvi}
\address{\AA bo Akademi University \\
Domkyrkotorget 1, 20500 \AA bo, Finland}
\email{neea.palojarvi@abo.fi}

\maketitle

\begin{abstract}
$\tau$-Li coefficients describe if a function satisfies the Generalized Riemann Hypothesis or not. In this paper we prove that certain values of the $\tau$-Li coefficients lead to existence or non-existence of certain zeros. The first main result gives explicit numbers $N_1$ and $N_2$ such that if all real parts of the $\tau$-Li coefficients are non-negative for all indices between $N_1$ and $N_2$, then the function has non zeros outside a certain region. According to the second result, if some of the real parts of the $\tau$-Li coefficients are negative for some index $n$ between numbers $n_1$ and $n_2$, then there is at least one zero outside a certain region. 

\end{abstract}

\section{Introduction}
\label{intoroduction}

In 1997 X.-J. Li \cite{li} proved an equivalent condition for the Riemann Hypothesis. The condition is based on the non-negativity of a real sequence $(\lambda_n)$, where
\begin{equation*}
		\lambda_n=\sum\limits_{\rho}\left(1-\left(1-\frac{1}{\rho}\right)^n \right)
\end{equation*} 
and the sum runs over the non-trivial zeros of the Riemann zeta function. The numbers $\lambda_n$ are called Li coefficients and they are non-negative if and only if the Riemann Hypothesis holds. Two years later E. Bombieri and J. C. Lagarias \cite{bombieri} proved that also other functions can be considered by using the Li coefficients. They also provided an arithmetic formula for the Li coefficients. There are several other results considering the Li coefficients for different sets of the functions. For example, J. C. Lagarias \cite{lagarias} investigated the Li coefficients for automorphic $L$-functions and L. Smajlovi\'{c} \cite{smajlovic} for a certain subclass of the extended Selberg class.

In 2006 P. Freitas \cite{freitas} proved that all the zeros of the Riemann zeta function lie inside the region $\Re(s)\le\frac{\tau}{2}$, where $\tau \in [\frac{1}{2}, \infty)$, if and only if the numbers
\begin{equation*}
		\frac{1}{\tau}\sum\limits_{\rho}\left(1-\left(\frac{\rho}{\rho-\tau}\right)^n \right)
\end{equation*}
are non-negative when $n\ge1$ is an integer. The sum runs over the non-trivial zeros of the Riemann zeta function and terms including zeros $\rho$ and $1-\rho$ are paired together. We notice that if $\tau=1$, then the condition is equivalent to the Li's condition. A. D. Droll \cite{droll} generalized the result for a certain subclass of the extended Selberg class. He proved that all the zeros of a function $F(s)$ in this subclass lie inside the region $\Re(s)\le \frac{\tau}{2}$ if and only if the terms
\begin{equation}
\label{formulaTauLi}
		\lambda_F(n,\tau)=\lim_{t \to \infty}\sum\limits_{\substack{\rho\\ |\Im(\rho)|\le t}}\left(1-\left(\frac{\rho}{\rho-\tau}\right)^n \right),
\end{equation}
where $\tau \in [1, 2)$ and the sum runs over the non-trivial zeros of the function $F(s)$, are non-negative for all positive integers $n$. These coefficients are called $\tau$-Li coefficients. A. Bucur, A.-M. Ernvall-Hyt\"onen, A. Od\v{z}ak and L. Smajlovi\'{c} \cite{bucur} also investigated the numerical behavior of some $\tau$-Li coefficients for some functions which violate the Riemann Hypothesis.

Zero-free regions in the critical strip can be investigated using the Li coefficients or the $\tau$-Li coefficients. F. C. Brown \cite[Theorem 3]{brown} proved that if a finite number of the Li coefficients for a certain function $F(s)$ are non-negative, then the critical strip contains zero-free regions. Furthermore, he also tried to show \cite[Theorem 2]{brown} that if there exist certain zero-free regions, then certain Li coefficients must be non-negative. Unfortunately, his proof of Lemma 5 contains two errors. A. D. Droll investigated the errors in his thesis and was able to fix one of them. As a result of the other error Brown's Theorem 2 is left unproved. 

In this paper we prove two main results regarding the $\tau$-Li coefficients and zero-free regions. The main results are summarized below: 

\begin{customthm}{\ref{mainresult1}}
Let $R>1$ be a real number. If all real parts of the $\tau$-Li coefficients attached to $F$ are non-negative in a certain interval, which is given in Theorem \ref{mainresult1}, then all zeros $\rho$ of $F$ satisfy the condition $\left|\frac{\rho}{\rho-\tau}\right|< R$. 
\end{customthm}

\begin{customthm}{\ref{mainresult2}}
Let $R>1$ be a real number. If at least one of the real parts of the $\tau$-Li coefficients attached to $F$ is negative in a certain interval, which is given in Theorem \ref{mainresult2}, then there is at least one zero $\rho$ of $F$ with $\left|\frac{\rho}{\rho-\tau}\right|\ge R$.
\end{customthm}

Theorem \ref{mainresult1} is similar to Brown's Theorem 3 and Theorem \ref{mainresult2} to Theorem 2 but in this article we consider also some cases $\tau\ne 1$ and have different conditions for the function $F(s)$. We do not make any assumptions of the order of a function $F(s)$ while Brown did. 

The results give a computational way to determine certain zero-free regions. Indeed, the results give us set of integers. Depending on the signs of the real parts of the $\tau$-Li coefficients computed using these integers it \textit{may} be possible to conclude that all zeros lie on certain region or there is at least one zero outside the region. If all real parts of the $\tau$-Li coefficients obtained from Theorem \ref{mainresult1} are non-negative, then there are no zeros $\rho$ with $\left|\frac{\rho}{\rho-\tau}\right|\ge R$. If some of the real parts are negative, there may or may not exists a zero $\rho$ with $\left|\frac{\rho}{\rho-\tau}\right|\ge R$. Next we can apply Theorem \ref{mainresult2}. If some of the real parts of the $\tau$-Li coefficients obtained from Theorem \ref{mainresult2} are negative, then there is at least one zero $\rho$ with $\left|\frac{\rho}{\rho-\tau}\right|\ge R$. The tricky case is that if some of the real parts obtained from Theorem \ref{mainresult1} are negative but all of the real parts obtained from Theorem \ref{mainresult2} are non-negative. In this case we can say nothing about the existence of the zeros $\rho$ with $\left|\frac{\rho}{\rho-\tau}\right|\ge R$. Instead, in this case there must be zeros $\rho$ with $\Re(s)>\frac{\tau}{2}$ since otherwise all real parts of the $\tau$-Li coefficients are non-negative.

Let $\tau >\frac{1}{e}$ be a real number. We are interested in the zeros $\rho$ which have $\Re(\rho) \in [0, \tau]$ and $\tau$-Li coefficients defined as
\begin{equation}
\label{formulaTauLiThis}
		\lambda_F(n,\tau)=\lim_{t \to \infty}\sum\limits_{\substack{\rho\\ |\Im(\rho)|\le t \\ 0\le \Re(\rho)\le \tau}}\left(1-\left(\frac{\rho}{\rho-\tau}\right)^n \right).
\end{equation}
With condition \ref{locationConditionMore} the previous definition coincides with the classical definition of the $\tau$-Li coefficient \eqref{formulaTauLi}.

In this paper we investigate a function $F(s)$ which satisfies the following conditions:
\begin{enumerate}[label=(\alph*)]
\label{conditionsForFunctions}
		\item\label{locationConditionMore} \textbf{Location of the zeros, 1}: The function $F(s)$ does not have zeros $\rho$ with $\Re(\rho)>\tau$.
		\item\label{locationCondition} \textbf{Location of the zeros, 2}: The function $F(s)$ does not have a zero $\rho=\tau$.
		\item\label{numberCondition} \textbf{Number of the zeros}: Let $\mathcal{N}_{F}(t)$ denote the number of the zeros $\rho$ of the function $F(s)$ with $0\le \Re(\rho) \le \tau$ and $0 \le |\Im(\rho)| \le t$. Furthermore, let similarly $\mathcal{N}_F(t_1, t_2)$ denote the number of the zeros $\rho$ of the function $F(s)$ with $0 \le \Re(\rho) \le \tau$ and $t_1< |\Im(\rho)| \le t_2$. For some real numbers $A_F>0$ and $B_F$ and for a real number $T_0>0$ which is large enough, we have the following two properties:
		\begin{equation}
		\label{CL}
		\begin{aligned}
				& \left|\mathcal{N}_{F}(T) -A_FT\log{T}-B_FT\right| \\
				& \quad<C_{F,1}(T_0)\log{T}+C_{F,2}(T_0)+\frac{C_{F,3}(T_0)}{T},
		\end{aligned}
		\end{equation}
		where $T\ge T_0$ is a real number and the numbers $C_{F,j}(T_0)$, where $j=1,2,3$, are non-negative real numbers which depend on the function $F(s)$ and the number $T_0$. Furthermore, we also have
		\begin{equation}
		\label{zeros24}
		\begin{aligned}
				&\left|\mathcal{N}_{F}(T,2T) -A_FT\log{T}-(A_F\log{4}+B_F)T\right| \\
				& \quad <c_{F,1}(T_0)\log{T}+c_{F,2}(T_0)+\frac{c_{F,3}(T_0)}{T},
		\end{aligned}
		\end{equation}
		where $c_{F,1}(T_0)$, $c_{F,2}(T_0)$ and $c_{F,3}(T_0)$ are non-negative real numbers. 

		We notice that formula \eqref{zeros24} actually follows from formula \eqref{CL}, indeed we have
		\begin{equation}
		\label{cFjBounds}
		\begin{aligned}
		       & c_{F,1}(T_0)\le 2C_{F,1}(T_0), \\
		       & c_{F,2}\le 2C_{F,2}(T_0)+C_{F,1}(T_0)\log{2}, \\
		       & c_{F,3}(T_0)\le\frac{3C_{F,3}(T_0)}{2}.   
		\end{aligned}        
		\end{equation} 
		However, formula \eqref{CL} does not follow from formula \eqref{zeros24}. Since we would like to apply both formula \eqref{CL} and formula \eqref{zeros24}, we have listed them separately.
		\item\label{computationCondition} \textbf{Computation}: The numbers $\lambda_F(n,\tau)$ can be computed without knowing the zeros of the function $F(s)$.
\end{enumerate}
The first condition is not necessary for proving the results. It is assumed since we would like keep our definition of the term $\lambda_F(n,\tau)$ \eqref{formulaTauLiThis} unchanged compared to the classical definition of the $\tau$-Li coefficient (see formula \eqref{formulaTauLi}).

Furthermore, the last condition is not needed to prove the results. Instead, it is needed for being able to apply the results to determine zero-free regions without knowing the zeros. For example, for $\tau=1$ and the Riemann zeta function the third condition is satisfied since by \cite{li} 
\begin{equation*}
		\lambda_\zeta(n, 1)=\frac{1}{(n-1)!}\frac{d^n}{ds^n}[s^{n-1}\log{\xi(s)}]_{s=1}
\end{equation*}
where
\begin{equation*}
		\xi(s)=s(s-1)\pi^{-\frac{s}{2}}\Gamma\left(\frac{s}{2}\right)\zeta(s).
\end{equation*}

The previous conditions are not very restrictive. For example, the Riemann zeta function satisfies these conditions. Furthermore, all Selberg class functions also satisfy these conditions for $\tau\ge 1$ (see \cite{droll,palojarvi,steuding}). 

Notice also that the assumption $\tau>\frac{1}{e}$ is a same type of assumption as Freitas did for the number $\tau$ and thus a very natural assumption. For example, non-trivial zeros $\rho$ of the Riemann zeta function satisfy $0<\Re(\rho)<1$. Since the number $\tau$ can be selected to be any real number in the interval $[1,2)$, all interesting values for the number $\tau$ are possible while considering the Riemann zeta function. The assumption $\tau>\frac{1}{e}$ is needed for technical purposes, indeed, mainly to be able to estimate $e\tau>1$ and thus $\log{(e\tau)}>0$ and $ne\tau>n$, where $n$ is a positive integer.

Our main goal is to prove Theorems \ref{mainresult1} and \ref{mainresult2}. We approach the proof as follows: We want to estimate the contribution of the zeros which lie outside of certain regions to the $\tau$-Li coefficients. By these estimates we can estimate the terms $\Re(\lambda_F(n,\tau))$ and obtain the results. For these results we need to prove some preliminary results in Section \ref{sectionContribution}. The main results are proved in Section \ref{sectionFindN}. In Section \ref{case1Zero} we consider a special case in which the function $F(s)$ has only one zero which lies outside of a certain region. 

At the end of the paper, in Sections \ref{subsDiri} and \ref{subsL}, we give numerical examples for Dirichlet $L$-functions and an $L$-function associated with a holomorphic newform. We keep in mind that there already are several results concerning Dirichlet $L$-functions and zero-free regions. Let $F(s)$ be a Dirichlet $L$-function associated with a primitive non-principal character modulo $q$. The first result concerning explicit zero-free regions was proved by K. S. McCurley \cite[Theorem 1]{mccurley}. According to the result, the function $F(s)$ has no zeros in the region
\begin{equation}
\label{mccurleyRegion}
    \Re(s)\ge 1-\frac{1}{9.645908801\log{(\max\{q,q|\Im(s)|,10\})}}
\end{equation}
to the exception of at most one zero. There are several improvements to the constant. Recently H. Kadiri \cite[Theorem 1.1.1]{kadiri} proved that the function $F(s)$ with $3\le q \le 400000$ does not vanish in the region
\begin{equation}
\label{kadiriRegion}
    \Re(s)=1-\frac{1}{5.60\log{(q\max\{1,|\Im(s)|\})}}.
\end{equation}
By F. C. Brown \cite[Corollary 1]{brown} we also know that for every $k  \ge 2$ there is at most one primitive Dirichlet character of conductor dividing $k$ such that the completed function of $F(s)$ has a zero $\rho$ with
$$
		\Re(\rho)\ge 1-\frac{1}{48\log{k}} \quad \text{and}\quad |\Im(\rho)|\le \frac{1}{48\log{k}}.
$$
If such a character exists, it is real and $\rho \in \mathbb{R}$ and $\rho$ is simple. K. Mazhouda \cite{mazhouda} investigated the non-negativity of the Li coefficients for the Dirichlet $L$-functions if the Generalized Riemann Hypothesis holds up to height $T$.

\section{Preliminary results} 
\label{sectionContribution}

In this section we prove upper bounds for the contributions of the zeros $\rho$ for the coefficients $\Re(\lambda_F(n,\tau))$. The main goal is to apply these results to determine the connections between the terms $\Re(\lambda_F(n,\tau))$ and the zeros of the function $F$. We consider two cases: first we consider the contribution of the zeros with the absolute values of the imaginary parts large enough, and then the contribution of the zeros with the absolute values of the imaginary parts small enough. We apply these results in Section \ref{sectionFindN}.

\subsection{Contribution of the zeros with the absolute values of the imaginary parts large enough}
\label{subsContributionLarge}

In this section we consider an upper bound for the contribution of the zeros $\rho$ whose imaginary parts have large enough absolute values to the $\tau$-Li coefficients. This means that we would like to find an upper bound for the term
\begin{equation}
\label{tauKerroinIso}
    \bigg |\lim_{t\to \infty}\sum_{T<|\Im(\rho)|\le t}\Re\left(1-\left(\frac{\rho}{\rho-\tau}\right)^n \right) \bigg |, 
\end{equation}
where $T$ is a (certain) positive real number, since
$$
		\lambda_F(n,\tau)=\lim_{t \to \infty}\sum\limits_{\substack{\rho\\ |\Im(\rho)|\le t}}\left(1-\left(\frac{\rho}{\rho-\tau}\right)^n \right)
$$ 
and we are interested in the term $\Re(\lambda_F(n,\tau))$. To estimate the contribution, we apply the binomial formula. We also assume that the number $n$ in the $\tau$-Li coefficient $\lambda_F(n,\tau)$ is large enough. This is not a very restrictive assumption since we would like to apply the result to find at least one number $n$ for which $\Re(\lambda_F(n,\tau))<0$, and this number $n$ can be large. 

\begin{theorem}
\label{largeZeros}
		Let $\tau >\frac{1}{e}$ be a real number and $T_0$, $A_F$, $B_F$, $c_{F,j}(T_0)$, where $j=1,2,3$, be defined as in condition \ref{numberCondition}. Assume that 
		$
				n \ge \max\left\{e,\frac{1}{e\tau}T_0\right\}
		$
 		is a positive integer and define $T(n)\coloneqq ne\tau$. Furthermore, let
		\begin{align*}
				K_{F,1}(\tau)&\coloneqq \frac{2\tau}{3}\left(e+\frac{1}{e}\right)\left(A_F+\left|A_F\log{(8e\tau)}+B_F\right|\right) \\
				& \quad +\frac{4}{27}\left(1+\frac{1}{e^2}\right)\left(\frac{c_{F,1}(T_0)}{3}\log{2} \right.\\
				&\quad\left.+c_{F,1}(T_0)\log{(e^2\tau)}+c_{F,2}(T_0)+\frac{2c_{F,3}(T_0)}{7e\tau}\right).
		\end{align*}
		Then
		\begin{align*}
				 \bigg |\lim_{t\to \infty}\sum_{T(n)<|\Im(\rho)|\le t}\Re\left(1-\left(\frac{\rho}{\rho-\tau}\right)^n \right) \bigg |< K_{F,1}(\tau)n\log{n}.
		\end{align*}	
\end{theorem}

\begin{proof}
		First we consider the term $\Re\left(1-\left(\frac{\rho}{\rho-\tau}\right)^n \right)$. We define
		\begin{equation}
		\label{defX}
			x\coloneqq \Re\left(\frac{\rho}{\rho-\tau}\right)=\frac{|\rho|^2-\Re(\rho)\tau}{|\rho-\tau|^2}=1+\frac{\Re(\rho)\tau-\tau^2}{|\rho-\tau|^2}
		\end{equation}
		and
		\begin{equation}
		\label{defY}
			y\coloneqq \Im\left(\frac{\rho}{\rho-\tau}\right)=-\frac{\Im(\rho)\tau}{|\rho-\tau|^2}.
		\end{equation}
		We write the term $\Re\left(1-\left(\frac{\rho}{\rho-\tau}\right)^n \right)$ using the numbers $x$ and $y$, but before that we estimate the numbers $x$ and $y$. Since 			$0\le\Re(\rho)\le \tau$, we have $\Re(\rho)\tau-\tau^2\le 0$. Furthermore, we also have $\left|\Re(\rho)\tau-\tau^2\right|\le \tau^2$ and $|\rho-\tau|^2\ge \Im(\rho)^2> \tau^2$ for $|\Im(\rho)|>T(n)$. Thus $|x|\le 1$. We also notice $|y|\le \frac{\tau}{|\Im(\rho)|}$.

		We have
		\[
		\begin{aligned}
			\Re\left(1-\left(\frac{\rho}{\rho-\tau}\right)^n \right) &=1-\Re\left(\sum\limits_{k=0}^n \binom{n}{k}x^{n-k}(iy)^k\right) \\
			& = 1-x^n-\sum\limits_{k=1}^{\floor{\frac{n}{2}}}\binom{n}{2k}x^{n-2k}(-1)^ky^{2k}.
		\end{aligned}
		\]
		We estimate this term in two parts: first we estimate the term $1-x^n$, and then the sum on the right-hand side of the previous formula. By the definition and the 	estimates for the term $x$ we have
		\[
			\left|1-x^n\right|\le\left|1-x\right|\sum\limits_{k=0}^{n-1}\left|x\right|^k\le \frac{\left|-\Re(\rho)\tau+\tau^2\right|}{|\rho-\tau|^2}n< \frac{n\tau^2}{\Im(\rho)^2}.
		\]

		Next, since $\binom{n}{2k}<\left(\frac{ne}{2k}\right)^{2k}$ \cite[inequality (2)]{das}, $|x|\le 1$ and $|y|\le \frac{\tau}{|\Im(\rho)|}$, we have
		\begin{align}
		\label{eqBinom}
				\left|\sum\limits_{k=1}^{\floor{\frac{n}{2}}}\binom{n}{2k}x^{n-2k}(-1)^ky^{2k}\right| &<\sum\limits_{k=1}^{\floor{\frac{n}{2}}}\left(\frac{ne\tau}{2k|					\Im(\rho)|}\right)^{2k} \nonumber \\
				& <\left(\frac{ne\tau}{2\Im(\rho)}\right)^2\sum\limits_{k=0}^\infty \left(\frac{ne\tau}{2\Im(\rho)}\right)^{2k} \nonumber\\
				& =\frac{(ne\tau)^2}{4\Im(\rho)^2}\cdot\frac{1}{1-\left(\frac{ne\tau}{2\Im(\rho)}\right)^2}.
		\end{align}
		Furthermore, since $T(n)=ne\tau$, for $|\Im(\rho)|>T(n)$ we have $1-\left(\frac{ne\tau}{2\Im(\rho)}\right)^2>\frac{3}{4}$. Thus the right-hand side of inequality \eqref{eqBinom} is
		\[
			<\frac{(ne\tau)^2}{3\Im(\rho)^2}
		\]
		and we have obtained that
		\[
				\left|\Re\left(1-\left(\frac{\rho}{\rho-\tau}\right)^n \right)\right| <\frac{3n\tau^2+(ne\tau)^2}{3\Im(\rho)^2}.
		\]

		Next we estimate term \eqref{tauKerroinIso} for $T=T(n)$. By the previous estimates we have
		\[
		\begin{aligned}
			& \bigg |\lim_{t\to\infty}\sum_{T(n)<|\Im(\rho)| \le t}\Re\left(1-\left(\frac{\rho}{\rho-\tau}\right)^n \right) \bigg | \\
			&\quad<\sum_{h=0}^{\infty} \sum_{|\Im(\rho)| \in (2^hT(n),2^{h+1}T(n)]} \frac{3n\tau^2+(ne\tau)^2}{3\Im(\rho)^2}.
		\end{aligned}
		\]
		From the assumptions posed on parameters $T(n)=ne\tau\ge T_0$ follows that we can apply formula \eqref{zeros24} for the number of the zeros and the right-hand side of the previous inequality is 
		\begin{align*}
			& <\sum\limits_{h=0}^\infty \frac{3n\tau^2+n^2e^2\tau^2}{3\cdot2^{2h}T(n)^2}\left(\vphantom{\frac{c_{F,3}(T_0)}{2^hT}}A_F2^{h}T(n)\log{\left(2^{h}T(n)\right)}\right.\\
			& \quad\left.+\left(A_F\log{4}+B_F\right)2^{h}T(n)+c_{F,1}(T_0)\log{\left(2^{h}T(n)\right)}+c_{F,2}(T_0)+\frac{c_{F,3}(T_0)}{2^hT(n)}\right) \\
			& \quad = \frac{3n\tau^2+n^2e^2\tau^2}{3T(n)}\left(\vphantom{\frac{c_{F,3}(T_0)}{2^hT}}2A_F\log{\left(2T(n)\right)}+2\left(A_F\log{4}+B_F\right) \right.\\
			& \quad\quad \left.+\frac{4c_{F,1}(T_0)}{9T(n)}\log{2}+\frac{4(c_{F,1}(T_0)\log{(T(n))}+c_{F,2}(T_0))}{3T(n)}+\frac{8c_{F,3}(T_0)}{7T(n)^2}\right) \\
			& \quad =\left(\frac{e\tau}{3}n+\frac{\tau}{e}\right)\left(\vphantom{\frac{c_{F,3}(T_0)}{2^hT}}2A_F\log{\left(8ne\tau\right)}+2B_F \right. \\
			& \quad\quad \left.+\frac{4c_{F,1}(T_0)}{9ne\tau}\log{2}+\frac{4(c_{F,1}(T_0)\log{(ne\tau)}+c_{F,2}(T_0))}{3ne\tau}+\frac{8c_{F,3}(T_0)}{7(ne\tau)^2}\right) \\
			& \quad =\frac{2e\tau A_F}{3}n\log{n}+\frac{2e\tau}{3}\left(A_F\log{(8e\tau)}+B_F\right)n+\frac{2\tau A_F}{e}\log{n}+\frac{4c_{F,1}(T_0)}{27}\log{2} \\
			& \quad\quad+\frac{4(c_{F,1}(T_0)\log{(ne\tau)}+c_{F,2}(T_0))}{9}+\frac{2\tau}{e}\left(A_F\log{(8e\tau)}+B_F\right)+\frac{8c_{F,3}(T_0)}{21ne\tau} \\
			& \quad\quad+\frac{4c_{F,1}(T_0)}{9ne^2}\log{2}+\frac{4(c_{F,1}(T_0)\log{(ne\tau)}+c_{F,2}(T_0))}{3ne^2}+\frac{8c_{F,3}(T_0)}{7n^2e^3\tau}.
		\end{align*}
		Since $n \ge e$ is an integer, the right-hand side of the previous formula is
		\[
		\begin{aligned}
			& \le n\log{n}\left(\frac{2e\tau A_F}{3}+\frac{2e\tau}{3}\left|A_F\log{(8e\tau)}+B_F\right|+\frac{2\tau A_F}{3e} +\frac{4c_{F,1}(T_0)}{81}\log{2}\right. \\
			& \left. \quad +\frac{4(c_{F,1}(T_0)\log{(e^2\tau)}+c_{F,2}(T_0))}{27}+\frac{2\tau}{3e}\left|A_F\log{(8e\tau)}+B_F\right|+\frac{8c_{F,3}(T_0)}{189e\tau}\right. \\
			& \left.\quad+\frac{4c_{F,1}(T_0)}{81e^2}\log{2}+\frac{4(c_{F,1}(T_0)\log{(e^2\tau)}+c_{F,2}(T_0))}{27e^2}+\frac{8c_{F,3}(T_0)}{189e^3\tau} \right) \\
			&\quad=K_{F,1}(\tau)n\log{n}.
		\end{aligned}
		\]
\end{proof}

\subsection{Contribution of the zeros with the absolute values of the imaginary parts small enough}
\label{subsContributionSmall}

In this section we investigate the upper bound for the contribution of the zeros with absolute values of the imaginary parts small enough to the coefficients $\Re\left(\lambda_F(n,\tau)\right)$. To obtain the result, we need the following lemma which is proved in Chapter $5$, Theorem $11$ in \cite{montgomery}:

\begin{lemma}
\label{min20}
		Let $M\ge1$ be an integer and let $z_1,z_2,\ldots, z_M$ be complex numbers which satisfy the condition $\max\limits_j |z_j|=1$. Then
		\begin{equation*}
				\max\limits_{1\le n \le 5M}\Re\left(\sum_{j=1}^Mz_j^n\right)\ge \frac{1}{20}.
		\end{equation*}
\end{lemma} 

Next we apply the above lemma and estimate the contribution of the zeros with the absolute values of the imaginary parts small enough. We also assume that the number $n$ in the $\tau$-Li coefficient $\lambda_F(n,\tau)$ is in a certain interval, and there exists a number $R>1$ such that for some zero $\rho$ it holds that $\left|\frac{\rho}{\rho-\tau}\right|\ge R$. These are not too restrictive assumptions since we would like to apply the result to find at least one number $n$ for which $\Re(\lambda_F(n,\tau))<0$. Using the following we can investigate whether there exists zeros outside certain regions or not. Recall also that $\Re(\rho)=\frac{\tau}{2}$ if and only if $\left|\frac{\rho}{\rho-\tau}\right|=1$.

Before going to the the next theorem, we would like to point out that we cannot remove any elements from the lower bound for the number $N$ described in the following theorem. We have made no other assumptions for the numbers $T_0$, $\tau$ and $R$ than they are greater than certain constants which are smaller than $e$, and the number $R$ can be arbitrary large. Furthermore, the number $M_F$ can be positive, negative or zero and the number $A_F$ can be any positive real number. Since we have also made no specific assumptions for the (numerical) relationships between the numbers $T_0$, $\tau$, $R$, $A_F$ and $B_F$ which would simplify the expressions described in the lower bound for the number $N$, we can conclude that no elements can be removed.

\begin{theorem}
\label{smallZeros}
Let $\tau >0$ be a real number and $T_0$, $A_F$, $B_F$, $C_{F,j}(T_0)$, where $j=1,2,3$, be defined as in condition \ref{numberCondition}. Assume that there is a real number $R>1$ such that there exists at least one zero $\rho$ with $\left|\frac{\rho}{\rho-\tau}\right|\ge R$. Furthermore, we define
\[
		M_F\coloneqq B_F+\frac{C_{F,1}(T_0)}{e}+\frac{C_{F,2}(T_0)}{3}+\frac{C_{F,3}(T_0)}{9}
\]
and assume that for an integer $N$ it holds
\[
N\ge \ceil[\Bigg]{\max\left\{e, T_0,\frac{\tau}{\sqrt{R^2-1}}, e^{\frac{1-15M_F}{15A_F}}\right\}}.
\]
Then
\[
\begin{aligned}
	& \Re\Bigg(\sum\limits_{\substack{\rho\\ |\Im(\rho)|\le N}}\left(1-\left(\frac{\rho}{\rho-\tau}\right)^n \right)\Bigg)<\sum\limits_{\substack{\rho\\ |\Im(\rho)|\le N}}1-\frac{1}{20}R^n
\end{aligned}
\]
for some positive integer $n \in[N, 5N^2(A_F\log{N}+M_F)]$ for which $N \mid n$.
\end{theorem}

\begin{proof}
First we prove that there exists an integer $n$ in $$[N, 5N^2(A_F\log{N}+M_F)]$$ such that $N \mid n$. Since we have $N \in \mathbb{Z}$, $N \ge 3$ and $N \ge  e^{\frac{1-15M_F}{15A_F}}$, we obtain
\[
5N(A_F\log{N}+M_F)\ge 15\left(A_F\cdot\frac{1-15M_F}{15A_F}+M_F\right)=1.
\]
Thus such an integer exists.

Next we prove the claim using Lemma \ref{min20}. First we recognize that we can apply Lemma \ref{min20} since there exist only finitely many zeros $\rho$ with $|\Im(\rho)|\le N$. By formula \eqref{CL}, for $N \ge \max\{3, T_0\}$ there are at most
\begin{equation*}
\begin{aligned}
		& A_FN\log{N}+B_FN+C_{F,1}(T_0)\log{N}+C_{F,2}(T_0)+\frac{C_{F,3}(T_0)}{N} \\
		& \quad\le N(A_F\log{N}+M_F)
\end{aligned}
\end{equation*}
zeros $\rho$ with $|\Im(\rho)|\le N$. Let these zeros be $\rho_1, \rho_2, \ldots, \rho_{M}$, where $M$ is a non-negative integer. Furthermore, since $N \ge \frac{\tau}{\sqrt{R^2-1}}$, for all zeros with $\left|\frac{\rho}{\rho-\tau}\right|\ge R$ it also holds that $\Re(\rho)\ge \frac{R\tau}{1+R}$ and
\[
|\Im(\rho)|\le \sqrt{\frac{\Re(\rho)^2-R^2\left(\Re(\rho)-\tau\right)^2}{R^2-1}}\le\frac{\tau}{\sqrt{R^2-1}}\le N.
\]
Thus, and by the assumptions for the number $R$, there also exists $R' \ge R$ such that $R'=\max\limits_{j\in [1,M]} |\frac{\rho_j}{\rho_j-\tau}|$ and $M \ge 1$. Thus we can set $\frac{\rho_j}{\rho_j-\tau}=R'r_j\exp(\phi_j i)$ ($j=1,2,\ldots, M$) where $0 \le r_j \le 1$ and $\phi_j$ are real numbers for all $j$. Then we can apply Lemma \ref{min20} for the complex numbers $z_j=r_j^Ne^{N\phi_j i}$ and get
\begin{equation*}
\begin{aligned}
		\Re\Bigg(\sum\limits_{\substack{\rho\\ |\Im(\rho)|\le N}}\left(1-\left(\frac{\rho}{\rho-\tau}\right)^n \right)\Bigg)&= \sum\limits_{j=1}^{M}\Re\left(1-\left(						\frac{\rho_j}{\rho_j-\tau}\right)^n \right) \\
		&=\sum\limits_{j=1}^{M} 1-R'^n\Re\left(\sum\limits_{j=1}^{M}z_j^\frac{n}{N}\right) \\
		&\le \sum\limits_{\substack{\rho\\ |\Im(\rho)|\le N }} 1-\frac{1}{20}R'^n
\end{aligned}
\end{equation*}
for some integer $\frac{n}{N} \in[1, 5M]$. Since $M< N(A_F\log{N}+M_F)$, the previous inequality holds for some integer $n \in[N, 5N^2(A_F\log{N}+M_F)]$ for which $N \mid n$. Furthermore, since $R\le R'$, the right-hand side of the previous inequality is
\[
\le \sum\limits_{\substack{\rho\\ |\Im(\rho)|\le N}} 1-\frac{1}{20}R^n,
\]
which we wanted to prove.
\end{proof}

\section{Main results}
\label{sectionFindN}

In this section we consider the terms $\Re(\lambda_F(n,\tau))$ and how they are related to zero-free regions. Recall that
$$
		\Re(\lambda_F(n,\tau))=\lim_{t \to \infty}\sum\limits_{\substack{\rho\\ |\Im(\rho)|\le t}}\Re\left(1-\left(\frac{\rho}{\rho-\tau}\right)^n \right).
$$ 
First we prove that if the terms $\Re(\lambda_F(n,\tau))$ are non-negative for all integers $n$ in a certain interval, then there are no zeros $\rho$ with $\left|\frac{\rho}{\rho-\tau}\right|\ge R$. The second theorem takes care of the other case: it states that if at least one term $\Re(\lambda_F(n,\tau))$ is negative for some integer $n$ in a certain interval, then there is at least one zero $\rho$ with $\left|\frac{\rho}{\rho-\tau}\right|> R$. The intervals will be given explicitly. Furthermore, the shape of the region $\left|\frac{\rho}{\rho-\tau}\right|<R$, where $R>1$, is coloured in white in Figure \ref{pic2}. The reason why we consider these kind of regions is that in Theorem \ref{smallZeros} we applied a lower bound for the term $\left|\frac{\rho}{\rho-\tau}\right|$. Furthermore, Brown also investigated these kind of regions. To prove the results we use the results proved in Section \ref{sectionContribution}. 

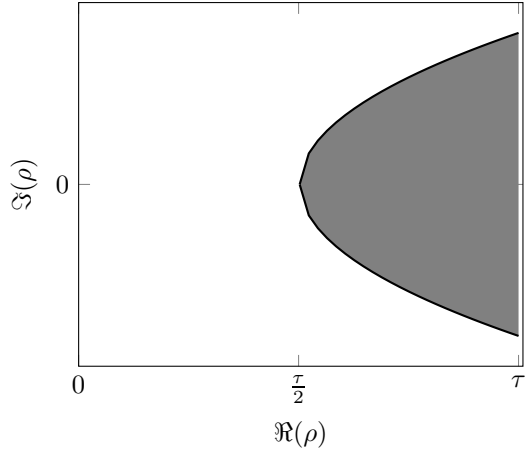
\begin{figure}[t]
\begin{center}
\begin{tikzpicture}
		\begin{axis}[width=0.55\textwidth, legend style={font=\small, at={(0.5,-0.15)},anchor=north},xlabel={$\Re(\rho)$}, ylabel={$\Im(\rho)$}, xmin=0, xmax=1.01, xtick={0,0.5,1}, xticklabels={0, $\frac{\tau}{2}$, $\tau$}, ytick={0}]
						\addplot+[name path=F3, mark=none, color=black, domain=0.502488:1, thick, forget plot] {sqrt((1.01*1.01*(x-1)*(x-1)-x*x)/(1-1.01*1.01))};
						\addplot+[name path=G3, mark=none, color=black, domain=0.502488:1, thick, forget plot] {-sqrt((1.01*1.01*(x-1)*(x-1)-x*x)/(1-1.01*1.01))};
						\addplot+[color=gray, forget plot]fill between[of=F3 and G3, soft clip={domain=0.502488:1}];
		\end{axis}
\end{tikzpicture}
\caption{Region $\left| \frac{\rho}{\rho-\tau}\right|<R$ is in white. The smaller the value $R$ is, the closer the left part of the gray region is to value $\tau/2$. Moreover, the smaller the number $R$ is, the larger the height of the gray region (meaning the largest absolute value of the imaginary parts in it) is.}
\label{pic2}
\end{center}
\end{figure}

In the next two results we use the inverse of the function $xe^x$. Let 
\[
W_0: [-e^{-1},\infty)\to [-1,\infty) \quad \text{and} \quad W_{-1}: [-e^{-1},0)\to(-\infty, -1]
\]
be different branches of the inverse of the function $xe^x$. For these branches we have  
\begin{align*}
& W_{0}(-e^{-1})=-1,\qquad\qquad\qquad \lim\limits_{x\to \infty}W_{0}(x)=\infty, \\
& W_{-1}(-e^{-1})=-1 \qquad \text{and}  \qquad \lim\limits_{x\to 0^{-}}W_{-1}(x)=-\infty.
\end{align*}
Furthermore, we remember that $\tau >\frac{1}{e}$ is a real number. Also, in this section, the terms $T_0$, $A_F$, $B_F$, $c_{F,j}(T_0)$ and $C_{F,j}(T_0)$, where $j=1,2,3$, are defined as in condition \ref{numberCondition}, $K_{F,1}(\tau)$ is defined as in Theorem \ref{largeZeros} and $M_F$ as in Theorem \ref{smallZeros}. Using this notation, let us define
\begin{equation}
\label{defK2}
	\begin{aligned}
		K_{F,2}(\tau)&\coloneqq 5(A_F+|M_F|)\left(\frac{5}{2}+\log{(5e\tau)}+\left|\log{\left(A_F+\frac{|M_F|}{1.732}\right)}\right|\right)\cdot \\
		& \quad \cdot \left(2\tau e^{\frac{5\tau^2M_F+2}{2}}\left(A_F+\frac{|B_F|}{\log{(3e\tau)}}+\frac{C_{F,1}(T_0)}{3e\tau}+\frac{C_{F,2}(T_0)}{3e\tau\log{(3e\tau)}} \right.\right. \\
		& \quad \left.\left.+\frac{C_{F,3}(T_0)}{9(e\tau)^2\log{(3e\tau)}}\right)+K_{F,1}(\tau)\right).
	\end{aligned}
\end{equation}

As before, we would like to point out that we cannot remove any elements from the lower bound for the number $N$ described in the following theorem. The reason behind this is that the elements in the lower bounds depend on different sets of constants.  

Now we are ready to prove the first main result:
\begin{theorem}
\label{mainresult1}
		Let $R>1$, $\tau >\frac{1}{e}$ be real numbers, $T_0$, $A_F$, $B_F$, $c_{F,j}(T_0)$, $C_{F,j}(T_0)$, where $j=1,2,3$, be defined as in condition \ref{numberCondition}, $M_F$ as in Theorem \ref{smallZeros} and $K_{F,2}(\tau)$ as in formula \eqref{defK2}. We define
		\begin{equation*}
				\begin{aligned}
						N &= \ceil[\Bigg]{\max\left\{\frac{\tau}{\sqrt{R^2-1}}, T_0, e^{\frac{1-15M_F}{15A_F}}, \right. \\
						& \quad \left. \exp\left(-W_{-1}\left(-\frac{4}{3(5\tau^2A_F+4)}\log{R}\right)\right), \frac{12\log{(20K_{F,2}(\tau))}}			{\log{R}}\right\}}
				\end{aligned}
		\end{equation*}  
		if $R \le e^{\frac{3(5\tau^2A_F+4)}{4e}}$ and 
		\begin{equation*}
				\begin{aligned}
						N &= \ceil[\Bigg]{\max\left\{e,\frac{\tau}{\sqrt{R^2-1}}, T_0, e^{\frac{1-15M_F}{15A_F}}, \frac{12\log{(20K_{F,2}(\tau))}}{\log{R}}\right\}}
				\end{aligned}
		\end{equation*}
		otherwise. 
		
		If all coefficients $\Re(\lambda_F(n,\tau))$ are non-negative for 
		\[
		    n \in [N,5N^2(A_F\log{N}+M_F)], \quad\text{where}\quad N \mid n,
		\]
		then all zeros $\rho$ satisfy the condition $\left|\frac{\rho}{\rho-\tau}\right|< R$.
\end{theorem}

\begin{proof}
		The main idea of the proof is to show that if there exists at least one zero $\rho$ with $\left|\frac{\rho}{\rho-\tau}\right|\ge R$, then $\Re(\lambda_F(n,\tau))$ is negative for some
		\[
		    n \in [N,5N^2(A_F\log{N}+M_F)], \quad\text{where}\quad N \mid n.
		\]
		
		First we notice that according to the first paragraph of the proof of Theorem \ref{smallZeros}, there exists an integer $n \in [N,5N^2(A_F\log{N}+M_F)]$ with $N \mid 			n$. Thus we can consider integers in $[N,5N^2(A_F\log{N}+M_F)]$.

		For all  $n \in [N,5N^2(A_F\log{N}+M_F)]$ we denote $T(n)\coloneqq ne\tau$. Clearly $T(n)>N$ and we have
		\begin{align}
				\Re(\lambda_F(n,\tau)) &=\lim_{t \to \infty}\sum_{\substack{T(n)<|\Im(\rho)|\le t}}\Re\left(1-\left(\frac{\rho}{\rho-\tau}\right)^n \right) \nonumber\\
				& \quad+\sum_{\substack{N<|\Im(\rho)|\le T(n)}}\Re\left(1-\left(\frac{\rho}{\rho-\tau}\right)^n \right) \label{realLambda}\\
				& \quad+\sum_{\substack{|\Im(\rho)|\le N}}\Re\left(1-						\left(\frac{\rho}{\rho-\tau}\right)^n \right) \nonumber.
		\end{align}
		The proof consists of estimating the three terms on the right-hand side of formula \eqref{realLambda}. 
		
		First we consider the second term. Remember that for the terms $x=\Re\left(\frac{\rho}{\rho-\tau}\right)$ and $y=\Im\left(\frac{\rho}{\rho-\tau}\right)$, defined in formulas \eqref{defX} and \eqref{defY}, we have proved that $|x| \le 1$
		and $|y|\le \frac{\tau}{|\Im(\rho)|}$ (see the first paragraph of the proof of Theorem \ref{largeZeros}). Thus 
		\[
				\left|\Re\left(\left(\frac{\rho}{\rho-\tau}\right)^n \right) \right|=\left|\Re\left(\left(x+yi\right)^n\right)\right|\le \left(1+\frac{\tau^2}{\Im(\rho)^2}\right)^\frac{n}{2}.
		\]
		Furthermore, since $n\le 5N^2(A_F\log{N}+M_F)$, for $|\Im(\rho)| > N$ the right-hand side is
		\[
		    < \left(1+\frac{\tau^2}{N^2}\right)^{\frac{N^2}{\tau^2}\cdot\frac{5\tau^2}{2}(A_F\log{N}+M_F)} <e^{\frac{5\tau^2M_F}{2}}N^{\frac{5\tau^2A_F}{2}}.
		\]

		By the previous estimate and Theorems \ref{largeZeros} and \ref{smallZeros}, the right-hand side of formula \eqref{realLambda} is
		\begin{equation}
		\label{firstStepSmaller}
		\begin{aligned}
				& <K_{F,1}(\tau)n\log{n}+e^{\frac{5\tau^2M_F}{2}}N^{\frac{5\tau^2A_F}{2}}\sum_{\substack{N<|\Im(\rho)|\le T(n)}} \left(e^{-\frac{5\tau^2M_F}{2}}					N^{-\frac{5\tau^2A_F}{2}}+1\right) \\
				& \quad+\sum\limits_{\substack{\rho\\ |\Im(\rho)|\le N}}1-\frac{1}{20} R^n 
		\end{aligned}
		\end{equation}
		for some integer $n \in [N,5N^2(A_F\log{N}+M_F)]$ for which $N \mid n$. We prove that the expression given in formula \eqref{firstStepSmaller} is less than zero for all $n \in [N,5N^2(A_F\log{N}+M_F)]$. First we consider the second and the third term of the formula \eqref{firstStepSmaller}. Thus we would like to estimate the term $e^{-\frac{5\tau^2M_F}{2}}N^{-\frac{5\tau^2A_F}{2}}$ and hence the term $M_F$. As we have already noticed, we have $5N(A_F\log{N}+M_F)\ge 1$ and thus 
		$$M_F\ge \frac{1}{5N}-A_F\log{N}>-A_F\log{N}.$$ 
		It follows that
		\begin{align*}
				& e^{\frac{5\tau^2M_F}{2}}N^{\frac{5\tau^2A_F}{2}}\sum_{\substack{N<|\Im(\rho)|\le T(n)}} \left(e^{\frac{-5\tau^2M_F}{2}}N^{\frac{-5\tau^2A_F}{2}}+1\right)+\sum\limits_{\substack{|\Im(\rho)|\le N}}1 \\
				& \quad < e^{\frac{5\tau^2M_F}{2}}N^{\frac{5\tau^2A_F}{2}}\left(\sum_{\substack{N<|\Im(\rho)|\le T(n)}} \left(1+1\right)+\sum\limits_{\substack{|\Im(\rho)|\le N}}	1\right) \\
				& \quad \le  e^{\frac{5\tau^2M_F}{2}}N^{\frac{5\tau^2A_F}{2}}\sum_{\substack{|\Im(\rho)|\le T(n)}} 2.
		\end{align*}
		By formula \eqref{CL} and since $T(n)=ne\tau\le 5e\tau N^2\left(A_F\log{N}+M_F\right)$ and $n\ge N\ge 3$, the previous formula is
		\begin{gather*}
		\begin{split}
				& < 2e^{\frac{5\tau^2M_F}{2}}N^{\frac{5\tau^2A_F}{2}}T(n)\log{T(n)}\left(A_F+\frac{B_F}{\log{T(n)}} \right. \\
				& \quad \left.+\frac{C_{F,1}(T_0)}{T(n)}+\frac{C_{F,2}(T_0)}{T(n)\log{T(n)}}+\frac{C_{F,3}(T_0)}{T(n)^2\log{T(n)}}\right) \\
				&\le 10\tau e^{\frac{5\tau^2M_F+2}{2}}N^{\frac{5\tau^2A_F+4}{2}}(A_F\log{N}+M_F)\log{\left(5e\tau N^2(A_F\log{N}+M_F)\right)}\cdot \\
				& \quad \cdot\left(A_F+\frac{|B_F|}{\log{(3e\tau)}}+\frac{C_{F,1}(T_0)}{3e\tau}+\frac{C_{F,2}(T_0)}{3e\tau\log{(3e\tau)}}+									\frac{C_{F,3}(T_0)}{9(e\tau)^2\log{(3e\tau)}}\right).
		\end{split}
		\end{gather*}
		Furthermore, since $\log{N}<\sqrt{N}$, the right-hand side of the previous inequality is
		\begin{equation}
		\label{eqSecondThird}
		\begin{aligned}
		& < 10\tau e^{\frac{5\tau^2M_F+2}{2}}N^{\frac{5\tau^2A_F+4}{2}}(A_F\log{N}+M_F)\left(\log{(5e\tau)}+2\log{N}+\log{\sqrt{N}}\right.\\
		& \quad \left.+\log{\left(A_F+\frac{|M_F|}{\sqrt{N}}\right)}\right)\cdot\left(A_F+\frac{|B_F|}{\log{(3e\tau)}} \right. \\
		& \quad \left. +\frac{C_{F,1}(T_0)}{3e\tau}+\frac{C_{F,2}(T_0)}{3e\tau\log{(3e\tau)}}+\frac{C_{F,3}(T_0)}{9(e\tau)^2\log{(3e\tau)}}\right) \\
		& < 10\tau e^{\frac{5\tau^2M_F+2}{2}}N^{\frac{5\tau^2A_F+4}{2}}(A_F\log{N}+M_F)\left(\frac{5}{2}\log{N}+\log{(5e\tau)}\right.\\
		& \quad \left.+\log{\left(A_F+\frac{|M_F|}{1.732}\right)}\right)\cdot\left(A_F+\frac{|B_F|}{\log{(3e\tau)}} \right. \\
		& \quad \left. +\frac{C_{F,1}(T_0)}{3e\tau}+\frac{C_{F,2}(T_0)}{3e\tau\log{(3e\tau)}}+\frac{C_{F,3}(T_0)}{9(e\tau)^2\log{(3e\tau)}}\right).
		\end{aligned}
		\end{equation}
		 We have estimated the second and the third term of formula \eqref{firstStepSmaller}.
		
		Next we consider the first and the last term of formula \eqref{firstStepSmaller}. Similarly as before, for $n \in [N,5N^2(A_F\log{N}+M_F)]$ and $N\ge 3$ we have $\log{N}<\sqrt{N}$ and
		\begin{gather}
		\label{eqK1RN}
		\begin{split}
					& K_{F,1}(\tau)n\log{n}-\frac{1}{20}R^n \\
					& <5N^2 K_{F,1}(\tau)(A_F\log{N}+M_F)\log{\left(5N^2(A_F\log{N}+M_F)\right)}-\frac{1}{20}R^N \\
					& < 5N^2 K_{F,1}(\tau)(A_F\log{N}+M_F)\cdot \\
					& \quad \cdot\left(\frac{5}{2}\log{N}+\log{5}+\log{\left(A_F+\frac{|M_F|}{1.732}\right)}\right)-\frac{1}{20}R^N.
		\end{split}
		\end{gather}
		Thus we have also estimated the first and last the term of formula \eqref{firstStepSmaller}.

		Now we can combine the previous computations and estimate formula \eqref{firstStepSmaller}. By estimates \eqref{eqSecondThird} and \eqref{eqK1RN}
		\begin{align}
				& K_{F,1}(\tau)n\log{n}+e^{\frac{5\tau^2M_F}{2}}N^{\frac{5\tau^2A_F}{2}}\sum_{\substack{N<|\Im(\rho)|\le T(n)}} \left(e^{-\frac{5\tau^2M_F}{2}}		N^{-\frac{5\tau^2A_F}{2}}+1\right) \nonumber\\
				& \quad+\sum\limits_{\substack{\rho\\ |\Im(\rho)|\le N}}1-\frac{1}{20} R^n  \nonumber\\
				& \quad <5N^{\frac{5\tau^2A_F+4}{2}}(A_F\log{N}+M_F)\cdot \label{eqSecondLF}\\
				& \quad\quad\cdot\left(\frac{5}{2}\log{N}+\log{(5e\tau)}+\log{\left(A_F+\frac{|M_F|}{1.732}\right)}\right)\cdot
				\label{eqSecondL} \\
				& \quad\quad \cdot \left(2\tau e^{\frac{5\tau^2M_F+2}{2}}\left(A_F+\frac{|B_F|}{\log{(3e\tau)}}+\frac{C_{F,1}(T_0)}{3e\tau}+\frac{C_{F,2}(T_0)}							{3e\tau\log{(3e\tau)}} \right.\right. \label{eqSecondLastL}\\
				& \quad\quad \left. \left. +\frac{C_{F,3}(T_0)}{9(e\tau)^2\log{(3e\tau)}}\right)+K_{F,1}(\tau)\right)-\frac{1}{20}R^N \label{eqLastL}.
		\end{align}	
		Since on lines \eqref{eqSecondLastL} and \eqref{eqLastL} in the previous inequality the only term which depends on the number $N$ is the term $R^N$, it is sufficient to estimate lines \eqref{eqSecondLF} and \eqref{eqSecondL} in the previous inequality. Since $N > e$, we have
		\begin{align*}
				& 5N^{\frac{5\tau^2A_F+4}{2}}(A_F\log{N}+M_F)\left(\frac{5}{2}\log{N}+\log{(5e\tau)}+\log{\left(A_F+\frac{|M_F|}{1.732}\right)}\right) \\
				& \quad < 5(A_F+|M_F|)N^{\frac{5\tau^2A_F+4}{2}}\log^2{N}\cdot\\
				& \quad\quad\cdot\left(\frac{5}{2}+\log{(5e\tau)}+\left|\log{\left(A_F+\frac{|M_F|}{1.732}\right)}\right|\right).
		\end{align*}
		Thus formula \eqref{firstStepSmaller} is
		\begin{equation*}
				< K_{F,2}(\tau)N^{\frac{5\tau^2A_F+4}{2}}\log^2{N}-\frac{1}{20}R^N.
		\end{equation*}

		We want to prove that the previous expression is at most zero for the number $N$. This can be equivalently written as
		\begin{equation}
		\label{smallZero}
		\begin{aligned}
				&  \left(\frac{5\tau^2 A_F}{4}+1\right)\frac{\log{N}}{N}+\frac{\log{\log{N}}}{N}+\frac{\log{(20K_{F,2}(\tau))}}{2N} \\
				& \quad \le\frac{1}{3}\log{R}+\frac{1}{8}\log{R}+\frac{1}{24}\log{R}.
		\end{aligned}
		\end{equation}
		We prove this in three parts. The coefficients $1/3, 1/8$ and $1/24$ are chosen because the term $\left(\frac{5\tau^2 A_F}{4}+1\right)\log{N}>\log{N}$ grows faster than the term $\log{\log{N}}$, this grows faster than a constant term and for all $N\ge e$ we have $3\log{N}\ge 8\log\log{N}$.

		First we prove that the first term on the left-hand side of inequality \eqref{smallZero} is at most the first term on the right-hand side of the inequality. This can be equivalently written as
 		\begin{equation}
		\label{eqMainLarger}
				-e^{-\log{N}}\log{N}\ge -\frac{4}{3(5\tau^2A_F+4)}\log{R}.
		\end{equation}          
		First we notice that if $R > e^{\frac{3(5\tau^2A_F+4)}{4e}}$, then for all $N \ge e$ we have
		\begin{equation*}
				-e^{-\log{N}}\log{N}=-\frac{\log{N}}{N}\ge -\frac{1}{e}> -\frac{4}{3(5\tau^2A_F+4)}\log{R}
		\end{equation*}
		and inequality \eqref{eqMainLarger} holds. Furthermore, since by the definition of the number $N$ given in formulation of the theorem we have
		\[
				N\ge \exp\left(-W_{-1}\left(-\frac{4}{3(5\tau^2A_F+4)}\log{R}\right)\right)\quad \text{if} \quad R \le e^{\frac{3(5\tau^2A_F+4)}{4e}}
		\]
		and the function $W_{-1}(x)$ is decreasing for $x \in (-1/e,0)$, inequality \eqref{eqMainLarger} holds also in this case and the inequality is proved. Morever, using estimate \eqref{eqMainLarger} for $N>1$ we also have
		\[
			\frac{\log{\log{N}}}{N}\le\frac{3}{8}\frac{\log{N}}{N}<\frac{3}{8}\left(\frac{5\tau^2 A_F}{4}+1\right)\frac{\log{N}}{N}\le \frac{3}{8}\cdot\frac{1}{3}\log{R}=\frac{1}{8}\log{R}.
		\]
		 Next we compare the third terms on the left- and right-hand side of inequality \eqref{smallZero}. Since $N \ge \frac{12\log{(20K_{F,2}(\tau))}}{\log{R}}$, we have $\frac{\log{(20K_{F,2}(\tau))}}{2N}\le \frac{1}{24}\log{R}$. Thus we have proved the claim.	
\end{proof}

\begin{remark}
We notice that the number $N$ in the previous theorem is not the best one. We used upper bounds for the terms instead of the exact values. On the other hand, the term $R^N$ grows faster than $N^{\frac{5\tau^2A_F+4}{2}}\log^2{N}$ and the smallest $N$ for which the term
$$
		K_{F,2}(\tau)N^{\frac{5\tau^2A_F+4}{2}}\log^2{N}-\frac{1}{20}R^N
$$
is non-positive depends mainly on the term $R$. Thus the number $N$ given by Theorem \ref{mainresult1} is good enough.
\end{remark}

Next we prove that if at least one of the real parts of the $\tau$-Li coefficients is negative in a certain interval, then there is at least one zero $\rho$ with $\left|\frac{\rho}{\rho-\tau}\right|\ge R$. To obtain the result, we use the notation where 
\begin{equation}
\label{defK3}
\begin{aligned}
		K_{F,3}(T,\tau)& \coloneqq \frac{0.432\tau^2}{T}\left(\frac{A_F}{2}\log{(2T)}+\frac{A_F\log{4}+B_F}{2} \right.\\
		& \quad \left.-\frac{c_{F,1}(T_0)}{3T}\log(2^\frac{1}{3}T) -\frac{c_{F,2}(T_0)}{3T}-\frac{2c_{F,3}(T_0)}{7T^2}\right ).
\end{aligned}
\end{equation}

In the following theorem we have an upper bound for the number $R$. This is not a too restrictive assumption since a smaller number $R$ means larger regions which do not have zeros and we would like to find as large areas as possible without zeros. Furthermore, we would like to notice that we cannot remove any elements from the lower bound for the term $T$. The reason is that most of the elements depend on different constants, and depending on the size of the term $\frac{B_F}{A_F}$ the term $\exp\left(-\frac{96B_F}{23A_F} \right)$ is sometimes greater than the term $\exp\left(\frac{0.324B_F}{(e^2-1.296)A_F}\right)$ and sometimes not. Furthermore, even though by equations \eqref{cFjBounds} we can bound the terms $c_{F,j}(T_0)$ ($j=1,2$) by the terms $C_{F,j}(T_0)$ ($j=1,2$), it is not enough for simplifying the lower bound for the term $T$.

Now we move on to the next theorem: 
\begin{theorem}
\label{mainresult2}
		Let $\tau >\frac{1}{e}$ be a real number and $T_0$, $A_F$, $B_F$, $c_{F,j}(T_0)$, $C_{F,j}(T_0)$, where $j=1,2,3$, be defined as in condition \ref{numberCondition} and $K_{F,3}(T,\tau)$ as in formula \eqref{defK3}. Furthermore, let $T$ be a real number for which 
		\begin{align*}
				T &> \max\left\{T_0, \exp\left(-\frac{96B_F}{23A_F} \right), \exp\left(\frac{0.324B_F}{(e^2-1.296)A_F}\right),  \right. \\
				& \quad\left. \frac{8c_{F,1}(T_0)}{3A_F}, \frac{8c_{F,2}(T_0)}{3A_FW_0\left(\frac{16c_{F,2}(T_0)}{3A_F}\right)}, \frac{96C_{F,1}(T_0)}{23A_F},  \right.\\
				&\left. \quad\frac{96C_{F,2}(T_0)}{23A_FW_0\left(\frac{96C_{F,2}(T_0)}{23A_F}\right)},\sqrt{\frac{192C_{F,3}(T_0)}{23A_FW_0\left(\frac{192C_{F,3}(T_0)}{23A_F}\right)}}\right\}.
		\end{align*}
		Further, let $R>1$ be a real number such that 
		\begin{equation*}
				R \le \exp\left(4W_0\left(\sqrt{\frac{K_{F,3}(T,\tau)}{4e^2\mathcal{N}_F(T)}}\right)\right).
		\end{equation*} 
		We denote
		\begin{equation*}
		\begin{aligned}
				n_0 &=\max\left\{1,\ceil[\Bigg]{\frac{1}{2}-\frac{2}{\log{R}}W_{0}\left(-\frac{\log{R}}{2}\sqrt{\frac{\mathcal{N}_F(T)}{K_{F,3}(T,\tau)}}			\exp\left(\frac{\log{R}}{4}\right)\right)}\right \}
		\end{aligned}
		\end{equation*}  
		and
		\begin{equation*}
		\begin{aligned}
				n_1&=\min\left\{\frac{T}{e\tau},\floor[\Bigg]{\frac{1}{2}-\frac{2}{\log{R}}W_{-1}\left(-\frac{\log{R}}{2}\sqrt{\frac{\mathcal{N}_F(T)}{K_{F,3}(T,\tau)}}								\exp\left(\frac{\log{R}}{4}\right)\right)}\right\}.
		\end{aligned}
		\end{equation*}  
		
		If the term $\Re(\lambda_F(n,\tau))$ is negative for some $n \in [n_0,n_1]$, then there exists at least one zero $\rho$ with $\left|\frac{\rho}{\rho-\tau}\right| \ge R$.
\end{theorem}

\begin{proof}
        For the proof the contraposition will be used. Thus we show that there are no zeros $\rho$ with $\left|\frac{\rho}{\rho-\tau}\right| \ge R$, then the terms $\Re(\lambda_F(n,\tau))$ are non-negative for all $n \in [n_0,n_1]$. 
        
		First we prove the case $n=1$. Since 
		\begin{equation*}
				\Re\left(1-\frac{\rho}{\rho-\tau}\right)=\frac{\tau(\tau-\Re(\rho))}{|\rho-\tau|^2}\ge 0,
		\end{equation*}
		the coefficient $\lambda_F(1,\tau)\ge 0$. Thus it is enough to consider the cases $n \ge 2$. We can estimate the $\tau$-Li coefficients by first considering the contribution of the zeros with absolute values of the imaginary parts greater than $T$ and then the contribution of zeros with the absolute values at most $T$. Indeed, we have
		\begin{equation}
		\label{liupper}
				\begin{aligned}
						\Re(\lambda_F(n,\tau ))&=\lim_{t\to \infty}\sum_{\substack{\rho \\ |\Im(\rho)|\le t}}\Re\left(1-\left(\frac{\rho}{\rho-\tau}\right)^n\right) \\
						&= \lim_{t\to \infty}\sum_{T<|\Im(\rho)|\le t}\Re\left(1-\left(\frac{\rho}{\rho-\tau}\right)^n\right)  \\
						& \quad+ \sum_{|\Im(\rho)|\le T}\Re\left(1-\left(\frac{\rho}{\rho-\tau}\right)^n\right).							
				\end{aligned}
		\end{equation}
		We estimate each term for $n \in [n_0,n_1]$.

		First we estimate the sum over the zeros with $|\Im(\rho)|>T$. Since
		\begin{align}
						\Re\left(1-\left(\frac{\rho}{\rho-\tau}\right)^n\right) &=\Re\left(1-\left(1+\frac{\tau}{\rho-\tau}\right)^n\right) \nonumber\\
						& = -n\tau\Re\left(\frac{1}{\rho-\tau}\right)-\frac{n(n-1)\tau^2}{2}\Re\left(\frac{1}{(\rho-\tau)^2}\right) \label{lower}\\
						&\quad-\sum\limits_{j=3}^n\binom{n}{j}\Re\left(\left(\frac{\tau}{\rho-\tau}\right)^j\right) \nonumber,
		\end{align}
		we can estimate each term on the right-hand side separately. First we estimate the first term on the right-hand side. Since $\Re(\rho)\le \tau$, we can compute
		\begin{equation*}
				\Re\left(\frac{1}{\rho-\tau}\right)=\frac{\Re(\rho)-\tau}{|\rho-\tau|^2}\le 0.
		\end{equation*}
		Thus the first term on the right-hand side of the equation \eqref{lower} is non-negative.

 		Next, we have
		\begin{equation*}
				\begin{aligned}
						\Re\left(\frac{1}{(\rho-\tau)^2}\right)=\frac{(\Re(\rho)-\tau)^2-\Im(\rho)^2}{|\rho-\tau|^4}.
				\end{aligned}
		\end{equation*}
		We notice that the term in the right-hand side of the previous equality is negative since $|\Im(\rho)|> ne\tau>\tau$. Thus we want to have large values in the denominator. Hence, the right-hand side of the previous equation is
		\begin{equation*}
				\begin{aligned}
						& \le \frac{\tau^2-\Im(\rho)^2}{(\tau^2+\Im(\rho)^2)^2}.	
				\end{aligned}
		\end{equation*}

		Finally we estimate the third term on the right-hand side of equation \eqref{lower}. Recall that $n\le n_1 \le \frac{T}{e\tau}$. Hence, for $j \ge 3$ and $|\Im(\rho)|>T$ we have
		\begin{equation*}
		\begin{aligned}
				\left|\binom{n}{j}\left(\frac{\tau}{\rho-\tau}\right)^j\right|&\le\left|\frac{n(n-1)(n-2)n^{j-3}}{j!}\left(\frac{\tau}{|\Im(\rho)|}\right)^j\right| \\
				&\le \frac{\tau^3n(n-1)(n-2)}{e^{j-3}j!|\Im(\rho)|^3}.
		\end{aligned}
		\end{equation*}
		Thus the third term on the right-hand side of equation \eqref{lower} is
		\begin{equation*}
				\begin{aligned}
						& \ge-\frac{\tau^3n(n-1)(n-2)}{|\Im(\rho)|^3}\sum_{j=3}^n\frac{1}{e^{j-3}j!} \\
						& >-\frac{\tau^3n(n-1)(n-2)}{|\Im(\rho)|^3}\sum_{j=3}^\infty \frac{1}{e^{j-3}j!} \\
						& >-\frac{0.184\tau^3n(n-1)(n-2)}{|\Im(\rho)|^3}.
				\end{aligned}
		\end{equation*}
		Using obtained bounds we can estimate the right-hand side of equation \eqref{lower}.

		By the above derived estimates we have
		\begin{equation}
		\label{eqfirstTogether}
		\begin{aligned}
				& \Re\left(1-\left(\frac{\rho}{\rho-\tau}\right)^n\right) \\
				&\quad>-\frac{n(n-1)\tau^2}{\Im(\rho)^2}\left(\frac{\tau^2-\Im(\rho)^2}{2(\frac{\tau^2}{|\Im(\rho)|}+|\Im(\rho)|)^2}+\frac{0.184(n-2)\tau}{|\Im(\rho)|}\right). 
		\end{aligned}		
		\end{equation}
		Further, since $|\Im(\rho)|>T\ge ne\tau$, we also have $\frac{0.184(n-2)\tau}{|\Im(\rho)|}<\frac{0.184}{e}$. The function
		\[
				\frac{\tau^2-\Im(\rho)^2}{2(\frac{\tau^2}{|\Im(\rho)|}+|\Im(\rho)|)^2}+\frac{0.184}{e}	
		\]
		is negative and decreasing in variable $\Im(\rho)$ for all $|\Im(\rho)|>T>ne\tau>2\tau$. This means that since the term $\frac{0.184(n-2)\tau}{|\Im(\rho)|}$ is a non-negative and decreasing function in variable $|\Im(\rho)|$, we can set $|\Im(\rho)|=ne\tau$ in the right-hand side of inequality \eqref{eqfirstTogether}.
		We obtain that the the right-hand side of formula \eqref{eqfirstTogether} is
		\begin{align*}
				&>-\frac{n(n-1)\tau^2}{\Im(\rho)^2}\left(\frac{1-(ne)^2}{2\left(\frac{1}{ne}+ne\right)^2}+\frac{0.184(n-2)}{ne}\ \right).
		\end{align*}
		The expression inside the brackets is negative for all $n \ge 2$. It has a stationary point $s \approx 2.837$ and it is an increasing function for $n> s$ and decreasing for $n \in [2,s]$. Thus the previous formula is
		\[
		\begin{aligned}
		    & \ge \frac{n(n-1)\tau^2}{\Im(\rho)^2}\min\left\{\frac{-1+(2e)^2}{2\left(\frac{1}{2e}+2e\right)^2},\lim_{n \to \infty}\left(\frac{-1+(ne)^2}{2\left(\frac{1}{ne}+ne			\right)^2}-\frac{0.184(n-2)}{ne} \right)\right\} \\
		    &>\frac{0.432n(n-1)\tau^2}{\Im(\rho)^2}.
		\end{aligned}
		\]
		Thus we have estimated the right-hand side of equation \eqref{lower}, and using this estimate we can estimate the first term on the right-hand side of equation \eqref{liupper}.
		
		Since the number of zeros in each interval $(2^hT,2^{h+1}T]$ is non-negative, by the previous estimates and formula \eqref{zeros24} we have
		\begin{align}
		\label{lifirst}
				& \lim_{t\to \infty}\sum_{\substack{\rho \\ T<|\Im(\rho)|\le t}}\Re\left(1-\left(\frac{\rho}{\rho-\tau}\right)^n\right) \nonumber \\
				& \quad >\sum\limits_{h=0}^\infty \sum\limits_{|\Im(\rho)|\in (2^hT,2^{h+1}T]} \frac{0.432n(n-1)\tau^2}{\Im(\rho)^2} \nonumber \\
				& \quad\ge\sum\limits_{h=0}^\infty \sum\limits_{|\Im(\rho)|\in (2^hT,2^{h+1}T]} \frac{0.432n(n-1)\tau^2}{2^{2h+2}T^{2}}  \nonumber \\
				& \quad > \frac{0.432n(n-1)\tau^2}{T}\sum\limits_{h=0}^\infty \left(\frac{A_F\log{(2^hT)}}{2^{h+2}}+\frac{A_F\log{4}+B_F}{2^{h+2}} \right.  \\
				& \quad\quad \left. -\frac{c_{F,1}(T_0)\log{(2^{h}T)}}{2^{2h+2}T}-\frac{c_{F,2}(T_0)}{2^{2h+2}T}-\frac{c_{F,3}(T_0)}{2^{3h+2}T^2} 									\vphantom{\frac{A_F\log{(2^hT)}}{2^{h+2}}}\right) \nonumber \\
				& \quad= \frac{0.432n(n-1)\tau^2}{T}\left(\frac{A_F}{2}\log{(2T)}+\frac{A_F\log{4}+B_F}{2} \right.\nonumber \\
				& \quad\quad \left.-\frac{c_{F,1}(T_0)}{3T}\log(2^\frac{1}{3}T) -\frac{c_{F,2}(T_0)}{3T}-\frac{2c_{F,3}(T_0)}{7T^2}\right) \nonumber \\
				& \quad=K_{F,3}(T,\tau)n(n-1) \nonumber.
		\end{align}
		Further, by the assumptions made in the formulation of the theorem, we have
		\begin{equation}
		\label{Tbound1}
		\begin{aligned}
		        T &> \max\left\{\exp\left(-\frac{96B_F}{23A_F} \right), \exp\left(\frac{0.324B_F}{(e^2-1.296)A_F}\right), \right. \\
				& \quad\left. \frac{8c_{F,1}(T_0)}{3A_F},  \frac{8c_{F,2}(T_0)}{3A_FW_0\left(\frac{16c_{F,2}(T_0)}{3A_F}\right)},
				\sqrt{\frac{192C_{F,3}(T_0)}{23A_FW_0\left(\frac{192C_{F,3}(T_0)}{23A_F}\right)}}\right\}.
		\end{aligned}		
		\end{equation}
		By direct computations and using the first two terms and the last term from the right-hand side of inequality \eqref{Tbound1} and the third inequality \eqref{cFjBounds} we also have
		\begin{equation*}
		\begin{aligned}
				T &> \max\left\{\frac{1}{2}\exp\left(-\frac{4}{A_F}(A_F\log{4}+B_F)\right), \sqrt{\frac{32c_{F,3}(T_0)}{7A_FW_0\left(\frac{128c_{F,3}(T_0)}{7A_F}\right)}}\right\}.
		\end{aligned}		
		\end{equation*}
		Thus we can divide the term $\frac{A_F}{2}\log{(2T)}$ from inequality \eqref{lifirst} by $4$, compare it to the other terms from inequality \eqref{lifirst}, use the previous lower bounds and obtain that the right-hand side of inequality \eqref{lifirst} is greater than zero. The reason why we divide by $4$ is that we do not know how large the terms $A_F$, $B_F$ and $c_{F,j}(T_0)$ ($j=1,2,3$) are and thus we do not know the optimal way to prove that the right-hand side of inequality \eqref{lifirst} is greater than zero. Hence, our choice is just simply divide by the number of the terms. This is not a crucial problem, see Remark \ref{remarkK3}.
		
		Next we estimate the second term on the right-hand side of inequality \eqref{liupper} and then combine the results. Since $|\frac{\rho}{\rho-\tau}|< R$ for all $\rho$, we have
		\begin{equation*}
				\begin{aligned}
						& \sum_{|\Im(\rho)|\le T}\Re\left(1-\left(\frac{\rho}{\rho-\tau}\right)^n\right)>\mathcal{N}_F(T)(1-R^n).
				\end{aligned}
		\end{equation*}
		Thus, and by inequalities \eqref{liupper} and \eqref{lifirst}, we have
		\begin{equation*}
					\lambda_F(n,\tau )>K_{F,3}(T,\tau)n(n-1)+\mathcal{N}_F(T)(1-R^n)
		\end{equation*}
		for $n \in [n_0,n_1]$. We want to prove that the right-hand side of the previous inequality is at least zero. 
		Equivalently, we can write
		\begin{equation}
		\label{eqK3step}
				\frac{K_{F,3}(T,\tau)}{\mathcal{N}_F(T)}\left(n-\frac{1}{2}\right)^2\ge R^n+\frac{K_{F,3}(T,\tau)}{4\mathcal{N}_F(T)}-1.
		\end{equation}
		We want to prove that 
		\begin{equation}
		\label{ineqK3}
		    \frac{K_{F,3}(T,\tau)}{4\mathcal{N}_F(T)}-1 \le 0
		\end{equation}
		because this implies that it is sufficient to consider the inequality
		\begin{equation}
		\label{wantGreater}
				\frac{K_{F,3}(T,\tau)}{\mathcal{N}_F(T)}\left(n-\frac{1}{2}\right)^2\ge R^n.
		\end{equation}
		
		Next we prove inequality \eqref{ineqK3}. First we prove $\mathcal{N}(T)\ge\frac{A_F}{24}T\log{T}$ and then, using this estimate, that inequality \eqref{wantGreater} holds. By formula \eqref{CL} it is sufficient to show that the following inequality holds:
		\begin{equation}
		\label{eq23AT}
				\frac{23}{24}A_FT\log{T}\ge-B_FT+C_{F,1}(T_0)\log{T}+C_{F,2}(T_0)+\frac{C_{F,3}(T_0)}{T}.
		\end{equation}
		Similarly as before, we can divide the term on the left-hand side by four and then compare it to the terms on the right-hand side. This is done because of the similar reasons as before i.e. since we do not know how large the terms $A_F$, $B_F$ and $C_{F,j}(T_0)$ ($j=1,2,3$) are. Since we have
		\[
		\begin{aligned}
				T&>\max\left\{\vphantom{ \sqrt{\frac{192C_{F,3}(T_0)}{23A_FW_0\left(\frac{192C_{F,3}(T_0)}{23A_F}\right)}}}
				\exp\left(-\frac{96B_F}{23A_F} \right), \frac{96C_{F,1}(T_0)}{23A_F}, \right. \\
				&\left.\quad \frac{96C_{F,2}(T_0)}{23A_FW_0\left(\frac{96C_{F,2}(T_0)}{23A_F}\right)},  \sqrt{\frac{192C_{F,3}(T_0)}{23A_FW_0\left(\frac{192C_{F,3}(T_0)}{23A_F}\right)}} \right\},
		\end{aligned}		
		\]
		and $W_0(x)$ is a decreasing function for $x>0$, inequality \eqref{eq23AT} holds. Thus we have $\mathcal{N}_F(T)\ge\frac{A_F}{24}T\log{T}$ and we can apply it to prove the estimate \eqref{eqK3step}.
		
		By the definition of the term $K_{F,3}(T,\tau)$ and since the terms $c_{F,j}(T_0)$, where $j=1,2,3$, are non-negative real numbers, to prove inequality \eqref{wantGreater} it is sufficient to show 
		\[
				0.432\tau^2\left(\frac{A_F}{2}\log{(2T)}+\frac{A_F\log{4}+B_F}{2}\right)\le \frac{1}{6}A_FT^2\log{T}.
		\]
		For all $T \ge ne\tau > 2$ and $n\ge 2$, we have
		\begin{equation}
		\label{eqSome2}
		    0.432\tau^2 \frac{A_F}{2}\log{(2T)}< \frac{0.432A_F}{(ne)^2} T^2\log{(T)} \le \frac{0.108A_F}{e^2}T^2\log{T}.
		\end{equation}
	    Furthermore, since we have also assumed 
		$
		        T>\exp\left(\frac{0.324B_F}{(e^2-1.296)A_F}\right),
		$
		we obtain
		\begin{equation}
		\label{eqSome}
		    	0.432\tau^2\frac{A_F\log{4}+B_F}{2}<\frac{0.432}{8e^2}T^2\left(2A_F\log{T}+B_F\right)< \left(\frac{1}{6}-\frac{0.108}{e^2}\right)A_FT^2\log{T}.
		\end{equation}
		Thus by combining inequalities \eqref{eq23AT}, \eqref{eqSome2} and \eqref{eqSome}, we have proved inequality \eqref{ineqK3}. To obtain the estimates for the terms $\Re\left(\lambda_F(n,\tau)\right)$, we need to prove inequality \eqref{wantGreater}.
		
		Taking a square root and dividing by $\exp\left(\frac{\log{R}}{2}\left(n-\frac{1}{2}\right)\right)$ ($n\ge 0.5$) and by $\sqrt{\frac{K_{F,3}(T,\tau) }{\mathcal{N}_F(T)}}$, inequality \eqref{wantGreater} can be equivalently written as
		\[
		\begin{aligned}
				\left(n-\frac{1}{2}\right)\exp\left(-\frac{\log{R}}{2}\left(n-\frac{1}{2}\right)\right)\ge \sqrt{\frac{\mathcal{N}_F(T)}{K_{F,3}(T,\tau)}}\exp\left(\frac{\log{R}}{4}\right).
		\end{aligned}		
		\]
		We multiply the inequality by $-\frac{\log{R}}{2}$ and obtain
		\[
		\begin{aligned}
				& -\frac{\log{R}}{2}\left(n-\frac{1}{2}\right)\exp\left(-\frac{\log{R}}{2}\left(n-\frac{1}{2}\right)\right) \\
				& \quad\le -\frac{\log{R}}{2}\sqrt{\frac{\mathcal{N}_F(T)}{K_{F,3}(T,\tau)}}\exp\left(\frac{\log{R}}{4}\right).
		\end{aligned}		
		\]
		This holds for $n \in [n_0,n_1]$, where
		\begin{equation*}
		\begin{aligned}
				n_0 &\ge \frac{1}{2}-\frac{2}{\log{R}}W_{0}\left(-\frac{\log{R}}{2}\sqrt{\frac{\mathcal{N}_F(T)}{K_{F,3}(T,\tau)}}		\exp\left(\frac{\log{R}}{4}\right)\right),
		\end{aligned}
		\end{equation*}  
		\begin{equation*}
		\begin{aligned}
				n_1& \le \frac{1}{2}-\frac{2}{\log{R}}W_{-1}\left(-\frac{\log{R}}{2}\sqrt{\frac{\mathcal{N}_F(T)}{K_{F,3}(T,\tau)}}								\exp\left(\frac{\log{R}}{4}\right)\right)
		\end{aligned}
		\end{equation*}
		and
		\begin{equation*}
				R \in \left(1,\exp\left(4W_0\left(\sqrt{\frac{K_{F,3}(T,\tau)}{4e^2\mathcal{N}_F(T)}}\right)\right)\right].
		\end{equation*} 
		
		Thus $\lambda_F(n,\tau )\ge 0$ also for any number $n \ge 2$ which lies in the interval $[n_0,n_1]$. This proves the claim.
\end{proof}

\begin{remark}
\label{remarkK3}
As in Theorem \ref{mainresult1}, the values of the numbers $n_0$ and $n_1$ in the previous theorem are not the best possible. They have been obtained by estimating the term $\lambda_F(n,\tau)$ and then estimating the result. Thus we lose precision. On the other hand, the term $R^n$ grows faster than the term $n(n-1)$ and thus the terms $n$ for which the expression
\begin{equation}
\label{formulaKR}
    K_{F,3}(T,\tau)n(n-1)+\mathcal{N}_F(T)(1-R^n)
\end{equation}
is non-negative mainly depend on the term $R$. 
\end{remark}

\begin{remark}
From formula \eqref{formulaKR} we also recognize that since
\[
\lim_{R \to 1} (1-R^n)=0,
\]
there always exist solutions for the terms $n_0$ and $n_1$ when the number $R$ is close enough to the number $1$.
\end{remark}

\section{Only one zero which lies outside of a certain region and a growth condition}
\label{case1Zero}

In this section, we consider relationships between growth conditions for the coefficients $\Re(\lambda_F(n, \tau))$ and existence of exactly one zero outside a certain region.

Besides of the non-negativity conditions, there are also growth conditions for the Li coefficients which imply the Generalized Riemann Hypothesis. For example, in 2006 A. Voros \cite{voros} proved that the Riemann Hypothesis is equivalent to the condition $\lambda_n \sim n(a\log{n}+b)$ with explicit $a>0$ and $b$. In 2010 and 2011 S. Omar and K. Mazhouda \cite{omar} and A. Od\v{z}ak and L. Smajlovi\'{c} \cite{odzak} derived similar conditions for certain classes containing the Selberg class. 

Furthermore, considering at most one zero outside a certain region is an interesting question. If we know that there exists at most one zero outside a certain region, then it is sufficient to consider how at most one zero affects the $\tau$-Li coefficients. These kind of results are known, for example, for Dirichlet $L$-functions, as we have already mentioned (recall results \eqref{mccurleyRegion},\eqref{kadiriRegion}).

We prove a growth condition for the coefficients $\Re(\lambda_F(n, \tau))$ whether the function $F(s)$ has exactly one zero $\rho_1$ with $|\frac{\rho_1}{\rho_1-\tau}|>1$ or not and we know a lower bound for the term $|\frac{\rho_1}{\rho_1-\tau}|$. The growth condition is same type as Voros' result i.e. $O(n\log{n})$. The result can be applied to the $\tau$-Li coefficients to determine whether the function $F(s)$ has exactly one zero outside a certain region or not. The advantages of the result compared to Theorems \ref{mainresult1} and \ref{mainresult2} are that we don't have to compute as many values $n$ as in those theorems. Furthermore, Theorem \ref{resultExp1} gives an equivalent condition while Theorems \ref{mainresult1} and \ref{mainresult2} do not.

Before we move on to the next theorem, let us introduce a new symbol in order to make the expressions little bit shorter. As before, let $\tau >\frac{1}{e}$ be a real number. The terms $T_0$, $A_F$, $B_F$ and $C_{F,j}(T_0)$, where $j=1,2,3$, are defined as in condition \ref{numberCondition} and the term $K_{F,1}(\tau)$ is defined as in Theorem \ref{largeZeros}. Using this notation, let us denote
\begin{equation}
\label{defK4}
\begin{aligned}
K_{F,4}(\tau)&\coloneqq 2\left(\vphantom{\frac{C_{F,3}(T_0)}{e^3\tau}}A_Fe\tau\log{(e^2\tau)}+|B_F|e\tau \right. \\
& \quad  \left.+\frac{C_{F,1}(T_0)\log{(e^2\tau)}+C_{F,2}(T_0)}{3}+\frac{C_{F,3}(T_0)}{9e\tau}\right).
\end{aligned}
\end{equation}

As before, we notice that we cannot remove any elements from the sets used in definitions for the term $N$. The elements depend on different constants. 

Now we can move on to the next theorem:

\begin{theorem}
\label{resultExp1}
Let $\tau >\frac{1}{e}$ be a real number and $T_0$, $A_F$, $B_F$, $c_{F,j}(T_0)$, $C_{F,j}(T_0)$, where $j=1,2,3$, be defined as in condition \ref{numberCondition}, $K_{F,1}(\tau)$ as in Theorem \ref{largeZeros} and $K_{F,4}(\tau)$ as in formula \eqref{defK4}. Suppose that the function $F(s)$ has at most one zero $\rho_1$ with $|\frac{\rho_1}{\rho_1-\tau}|>1$. Furthermore, we also assume that if such a zero $\rho_1$ exists, then $R>1$ is a real number such that $|\frac{\rho_1}{\rho_1-\tau}|\ge R$. 
Let
\[
\begin{aligned}
N &= \ceil[\Bigg]{\max\left\{\frac{1}{e\tau}T_0, \exp\left(-W_{-1}\left(-\frac{2\log{R}}{3}\right)\right), \right.\\
&\quad \left.\frac{12\log{\left(40\left(0.5+K_{F,1}(\tau)+K_{F,4}(\tau)\right)\right)}}{\log{R}}\right\}}
\end{aligned}
\]
if $R \le e^{\frac{3}{2e}}$, and
\[
N = \ceil[\Bigg]{\max\left\{e, \frac{1}{e\tau}T_0, \frac{12\log{\left(40\left(0.5+K_{F,1}(\tau)+K_{F,4}(\tau)\right)\right)}}{\log{R}}\right\}}
\]
otherwise. 

The zero $\rho_1$ exists if and only if 
\[
|\Re(\lambda_F(n, \tau))|\ge \left(K_{F,1}(\tau)+K_{F,4}(\tau)\right)n\log{n}
\]
for at least one integer $n \in [N, 5N]$ where $N \mid n$.
\end{theorem}

\begin{proof}
The proof consists of finding an upper bound for the term $|\Re(\lambda_F(n, \tau))|$ if the zero $\rho_1$ does not exist and a lower bound for the term $|\Re(\lambda_F(n,\tau))|$ if the zero exists. These results can be derived from the same estimates.

We denote $T(n)\coloneqq ne\tau$. For all zeros $\rho \ne \rho_1$ of the function $F(s)$ we have $|\frac{\rho}{\rho-\tau}|\le 1$ and thus $\Re(\rho) \le \frac{\tau}{2}$. Hence we have
\begin{equation}
\label{formulaException}
		\begin{aligned}
				\Re(\lambda_F(n,\tau)) &=\lim_{t \to \infty}\sum_{\substack{T(n)<|\Im(\rho)|\le t \\ 0 \le \Re(\rho)\le \frac{\tau}{2}}}\Re\left(1-\left(\frac{\rho}{\rho-\tau}	\right)^n \right)\\
				& \quad +\sum_{\substack{|\Im(\rho)|\le T(n) \\ 0 \le \Re(\rho)\le \frac{\tau}{2}}}\Re\left(1-\left(\frac{\rho}{\rho-\tau}\right)^n \right)+\Re\left(1-		\left(\frac{\rho_1}{\rho_1-\tau}\right)^n\right),
		\end{aligned}
\end{equation}
where the last term exists if and only if the zero $\rho_1$ exists. The first term on the right-hand side can be estimated by Theorem \ref{largeZeros}. Thus it is sufficient to estimate the last two terms of the previous equation.

We have
\begin{align*}
\left|\sum_{\substack{|\Im(\rho)|\le T(n) \\ 0 \le \Re(\rho)\le \frac{\tau}{2}}}\Re\left(1-\left(\frac{\rho}{\rho-\tau}\right)^n \right)\right| \le \sum_{\substack{|\Im(\rho)|\le T(n) \\ 0 \le \Re(\rho)\le \frac{\tau}{2}}}2.
\end{align*}
By formula \eqref{CL}, for $n \ge N \ge 3$ the right-hand side is 
\begin{align*}
& < 2\left(\vphantom{\frac{C_{F,3}(T_0)}{T(n)}}A_FT(n)\log{T(n)}+B_FT(n) \right.\\
& \left.\quad+C_{F,1}(T_0)\log{T(n)}+C_{F,2}(T_0)+\frac{C_{F,3}(T_0)}{T(n)}\right) \\
& \le 2n\log{n}\left(\vphantom{\frac{C_{F,3}(T_0)}{e^3\tau}}A_Fe\tau\log{(e^2\tau)}+|B_F|e\tau \right. \\
& \quad  \left.+\frac{C_{F,1}(T_0)\log{(e^2\tau)}+C_{F,2}(T_0)}{3}+\frac{C_{F,3}(T_0)}{9e\tau}\right) \\
& = K_{F,4}(\tau)n\log{n}.
\end{align*}
Using formula \eqref{formulaException} and Theorem \ref{largeZeros}, it follows
\[
|\Re(\lambda_F(n,\tau))|<K_{F,1}(\tau)n\log{n}+K_{F,4}(\tau)n\log{n}
\]
if the function $F(s)$ does not have zeros with $\Re(\rho)>\frac{\tau}{2}$.

We have also almost estimated the right-hand side of formula \eqref{formulaException} with the zero $\rho_1$. Similarly as in the proof of Theorem \ref{smallZeros}, we obtain 
\begin{equation*}
		\Re\left(1-\left(\frac{\rho_1}{\rho_1-\tau}\right)^n\right) \le 1-\frac{1}{20}R^n
\end{equation*}
for some integer $n \in [N,5N]$ for which $N \mid n$. For $n \ge N \ge \frac{\log{20}}{\log{R}}$, this means that
\[
\left|\Re\left(1-\left(\frac{\rho_1}{\rho_1-\tau}\right)^n\right)\right| \ge \frac{1}{20}R^n-1. 
\]
Thus, if a zero $\rho$ exists, then for some $n \in [N,5N]$ it holds that
\begin{equation*}
|\Re(\lambda_F(n,\tau))|> \frac{1}{20}R^n-1-\left(K_{F,1}(\tau)+K_{F,4}(\tau)\right)n\log{n}.
\end{equation*}
We want to prove that the previous formula is at least $$(K_{F,1}(\tau)+K_{F,4}(\tau))n\log{n}$$ for all $n \in [N, 5N]$.

It is sufficient to show
\[
R^n\ge 40n\log{n}(\frac{1}{2}+K_{F,1}(\tau)+K_{F,4}(\tau)).
\]
This can be equivalently written as
\[
\begin{aligned}
& \frac{2}{3}\log{R}+\frac{1}{4}\log{R}+\frac{1}{12}\log{R} \\
&\quad\ge \frac{\log{n}}{n}+\frac{\log{\log{n}}}{n}+\frac{\log{\left(40\left(0.5+K_{F,1}(\tau)+K_{F,4}(\tau)\right)\right)}}{n}.
\end{aligned}
\]
This follows similarly from the assumptions for the number $n$ as result \eqref{smallZero} in the proof of Theorem \ref{mainresult1}. Also, the coefficients $2/3$, $1/4$ and $1/12$ are selected because of the similar reasons as in the proof of Theorem \ref{mainresult1}. Indeed, the term $\log{n}$ grows faster than the term $\log{\log{n}}$, this grows faster than a constant term and for all $n\ge e$ we have $3\log{n}\ge 8\log\log{n}$.

Thus we have proved the claim.
\end{proof}

\section{Example: Dirichlet \texorpdfstring{$L$}{}-functions}
\label{subsDiri}

Let $F(s)$ be a Dirichlet $L$-function associated with a primitive non-principal character modulo $q$ and $\tau \ge 1$. We know that the function $F$ does not have zeros with real parts greater than one and $F(1)\ne 0$. Thus it satisfies conditions \ref{locationConditionMore} and \ref{locationCondition} for $\tau$. By T. S. Trudgian \cite[Theorem 1]{trudgian}, for $T\ge 1$ we have 
\begin{equation}
\label{diricformula}
\begin{aligned}
		&\left|\mathcal{N}_{F}(T) -\frac{T}{\pi}\log{T}-\frac{T}{\pi}\log{\frac{q}{2\pi e}}\right|<0.317\log{T}+0.317\log{q}+6.401.
\end{aligned}
\end{equation} 
Thus the function $F(s)$ also satisfies condition \ref{numberCondition}. Furthermore, the function $F(s)$ is in the Selberg class and thus by \cite[Lemma 2.1.2]{droll} satisfies also condition \ref{computationCondition}. 

We can set
$$
		\begin{aligned}
				& A_F=\frac{1}{\pi}, \quad B_F=\frac{1}{\pi}\log{\frac{q}{2\pi e}}, \quad T_0=1, \\
				& C_{F,1}(T_0)=0.317, \quad  c_{F,1}(T_0)=0.634, \\
				& C_{F,2}(T_0)=0.317\log{q}+6.401, \\
				& c_{F,2}(T_0)=0.317\log{2}+0.634\log{q}+12.802, \\
				& \text{and} \quad C_{F,3}(T_0)=c_{F,3}(T_0)=0. \\
		\end{aligned}
$$ 
This leads to the following two corollaries which describe Theorems \ref{mainresult1} and \ref{resultExp1} for the function $F(s)$. We do not write a full corollary for the result obtained from Theorem \ref{mainresult2} since the formulas are quite long and we give little bit nicer numerical results in Section \ref{subsubsDiriSecond}.

The first one describes the relationship between Theorem \ref{mainresult1} and the Dirichlet $L$-functions associated with a primitive non-principal character modulo $q$. We cannot remove any terms form the sets which give lower bound for the number $N$ since the term $K_{F,2}(\tau)$ depends on the number $q$ on the other terms do not and the first term may be larger than the other ones.
\begin{corollary}
\label{corolDiri1}
Let $R>1$ and $\tau\ge 1$ be real numbers. If we consider a Dirichlet $L$-function associated with a primitive non-principal character modulo $q$, then in Theorem \ref{mainresult1} we have
	\begin{equation*}
				\begin{aligned}
						N &= \ceil[\Bigg]{\max\left\{\frac{\tau}{\sqrt{R^2-1}}, \right. \\
						& \quad \left. \exp\left(-W_{-1}\left(-\frac{4}{3(5\tau^2/\pi+4)}\log{R}\right)\right), \frac{12\log{(20K_{F,2}(\tau))}}			{\log{R}}\right\}}
				\end{aligned}
		\end{equation*}  
		if $R \le e^{\frac{3(5\tau^2/\pi+4)}{4e}}$ and 
		\begin{equation*}
				\begin{aligned}
						N &= \ceil[\Bigg]{\max\left\{e,\frac{\tau}{\sqrt{R^2-1}}, \frac{12\log{(20K_{F,2}(\tau))}}{\log{R}}\right\}}.
				\end{aligned}
		\end{equation*}
		
		Here
		\begin{equation*}
    	\begin{aligned}
		K_{F,2}(\tau)&= 5(1/\pi+M_F)\left(\frac{5}{2}+\log{(5e\tau)}+\log{\left(1/\pi+\frac{M_F}{1.732}\right)}\right)\cdot \\
		& \quad \cdot \left(2\tau e^{\frac{5\tau^2M_F+2}{2}}\left(\frac{1}{\pi}+\frac{\left|\log{\frac{q}{2\pi e}}\right|}{\pi\log{(3e\tau)}}+\frac{0.317}{3e\tau}\right.\right. \\
		& \quad \left.\left.+\frac{0.317\log{q}+6.401}{3e\tau\log{(3e\tau)}} \right)+K_{F,1}(\tau)\right),
	\end{aligned}
    \end{equation*}
		\[
		M_F= \frac{1}{\pi}\log{\frac{q}{2\pi e}}+\frac{0.317}{e}+\frac{0.317\log{q}+6.401}{3}
        \]
        and
        \begin{align*}
				K_{F,1}(\tau)&= \frac{2\tau}{3\pi}\left(e+\frac{1}{e}\right)\left(1+\left|\log{\frac{4q\tau}{\pi}}\right|\right) \\
				& \quad +\frac{4}{27}\left(1+\frac{1}{e^2}\right)\left(\frac{0.634}{3}\log{2}+0.317\log{(2e^4q^2\tau^2)}+12.802\right).
		\end{align*}
\end{corollary}

Now, we describe what the \ref{resultExp1} says for the function $F(s)$. Please notice that in the next theorem we cannot remove any element from the sets since they depend on the different sets of variables.
\begin{corollary}
\label{coroDiri2}
Let us consider a Dirichlet $L$-function associated with a primitive non-principal character modulo $q$ and let $R>1$ and $\tau\ge 1$ be real numbers. Then the number $N$ defined in Theorem \ref{resultExp1} can be written as
\[
\begin{aligned}
N &= \ceil[\Bigg]{\max\left\{\exp\left(-W_{-1}\left(-\frac{2\log{R}}{3}\right)\right), \right.\\
&\quad \left.\frac{12\log{\left(40\left(0.5+K_{F,1}(\tau)+K_{F,4}(\tau)\right)\right)}}{\log{R}}\right\}}
\end{aligned}
\]
if $R \le e^{\frac{3}{2e}}$, and
\[
N = \ceil[\Bigg]{\max\left\{e, \frac{12\log{\left(40\left(0.5+K_{F,1}(\tau)+K_{F,4}(\tau)\right)\right)}}{\log{R}}\right\}}
\]
otherwise.

Here $K_{F,1}(\tau)$ is defined as on Corollary \ref{corolDiri1} and
\begin{equation*}
\begin{aligned}
K_{F,4}(\tau)&= 2\left(\frac{e\tau}{\pi}\left(\log{(e^2\tau)}+\left|\log{\frac{q}{2\pi e}}\right|\right) +\frac{0.317\log{(e^2q\tau)}+6.401}{3}\right).
\end{aligned}
\end{equation*}
\end{corollary}

In the next two sections we consider numerical examples of Theorems \ref{mainresult1}, \ref{mainresult2} and \ref{resultExp1} and Corollaries \ref{corolDiri1} and \ref{coroDiri2} for Dirichlet $L$-functions associated with a primitive non-principal character modulo $q=100$.  

We consider different regions $\left|\frac{\rho}{\rho-\tau}\right|\ge R$ determined by the numbers $R$ and $\tau$. Some regions $\left|\frac{\rho}{\rho-\tau}\right|\ge R$ for $\tau=1$ and different values for the number $R$ are described in Figure \ref{pic1}. The regions are symmetric with respect to the line $\Re(s)=\frac{1}{2}$ since the zeros of the function $F(s)$ lie symmetrically with respect to this line. Figure \ref{pic1} also contains the results described in formulas \eqref{mccurleyRegion} and \eqref{kadiriRegion} and proved by McCurley \cite[Theorem 1]{mccurley} and Kadiri \cite[Theorem 1.1.1]{kadiri} respectively. As it can be seen in Figure \ref{pic1}, the results proved in Theorems \ref{mainresult1} and \ref{mainresult2} allow us to consider existence of the zeros which have real parts close to the line $\Re(s)=\frac{1}{2}$ while Kadiri's and McCurley's results consider zeros which real parts are close to the lines $\Re(s)=1$ and $\Re(s)=0$. On the other hand, Kadiri's and McCurley's results do not have upper bound for the absolute values of imaginary parts of the zeros whereas Theorems \ref{mainresult1}, \ref{mainresult2} and \ref{resultExp1} consider only zeros up to some height. Kadiri's and McCurley's results also provide clear zero-free regions while the results proved in this article only provide some conditions to hold. Also, with the exception of Theorem \ref{resultExp1}, we do \textbf{not} give any if and only if statements.

\begin{figure}
\begin{tikzpicture}
		\begin{axis}[width=0.75\textwidth, legend style={font=\small, at={(1.05,0.5)},anchor=west},xlabel={$\Re(\rho)$}, ylabel={$\Im(\rho)$}]
						\addplot+[name path=F7,mark=none, color=black, domain=0.50025:1, forget plot] {sqrt((1.0001*1.0001*(x-1)*(x-1)-x*x)/											(1-1.0001*1.0001))};
						\addplot+[name path=G7, mark=none, color=black, domain=0.50025:1, forget plot] {-sqrt((1.0001*1.0001*(x-1)*(x-1)-x*x)/											(1-1.0001*1.0001))};
						\addplot+[name path=F8, mark=none, color=black, domain=0:0.49975, forget plot] {sqrt((1.0001*1.0001*(-x)*(-x)-(1-x)*(1-x))/										(1-1.0001*1.0001))};
						\addplot+[name path=G8, mark=none, color=black, domain=0:0.49975, forget plot] {-sqrt((1.0001*1.0001*(-x)*(-x)-(1-x)*(1-x))/										(1-1.0001*1.0001))};
						\addplot+[pattern=dots, pattern color=gray, forget plot]fill between[of=F7 and G7, soft clip={domain=0.50025:1}];
						\addplot[pattern=dots, pattern color=gray]fill between[of=F8 and G8, soft clip={domain=0:0.49975}];
						\addlegendentry{$R=1.0001$}

						\addplot+[name path=F5, mark=none, color=black, domain=0.50025:1, forget plot] {sqrt((1.001*1.001*(x-1)*(x-1)-x*x)/											(1-1.001*1.001))};
						\addplot+[name path=G5, mark=none, color=black, domain=0.50025:1, forget plot] {-sqrt((1.001*1.001*(x-1)*(x-1)-x*x)/											(1-1.001*1.001))};
						\addplot+[name path=F6, mark=none, color=black, domain=0:0.49975, forget plot] {sqrt((1.001*1.001*(-x)*(-x)-(1-x)*(1-x))/										(1-1.001*1.001))};
						\addplot+[name path=G6, mark=none, color=black, domain=0:0.49975, forget plot] {-sqrt((1.001*1.001*(-x)*(-x)-(1-x)*(1-x))/										(1-1.001*1.001))};
						\addplot+[pattern=horizontal lines, pattern color=gray, forget plot]fill between[of=F5 and G5, soft clip={domain=0.50025:1}];
						\addplot[pattern=horizontal lines, pattern color=gray]fill between[of=F6 and G6, soft clip={domain=0:0.49975}];
						\addlegendentry{$R=1.001$}

						\addplot+[name path=F3, mark=none, color=black, domain=0.502488:1, thick, forget plot] {sqrt((1.01*1.01*(x-1)*(x-1)-x*x)/(1-1.01*1.01))};
						\addplot+[name path=G3, mark=none, color=black, domain=0.502488:1, thick, forget plot] {-sqrt((1.01*1.01*(x-1)*(x-1)-x*x)/										(1-1.01*1.01))};
						\addplot+[name path=F4, mark=none, color=black, domain=0:0.497512, thick, forget plot] {sqrt((1.01*1.01*(-x)*(-x)-(1-x)*(1-x))/										(1-1.01*1.01))};
						\addplot+[name path=G4, mark=none, color=black, domain=0:0.497512, thick, forget plot] {-sqrt((1.01*1.01*(-x)*(-x)-(1-x)*(1-x))/									(1-1.01*1.01))};
						\addplot+[pattern=grid,  pattern color=gray, forget plot]fill between[of=F3 and G3, soft clip={domain=0.502488:1}];
						\addplot[pattern=grid,  pattern color=gray]fill between[of=F4 and G4, soft clip={domain=0:0.497512}];
						\addlegendentry{$R=1.01$}

						\addplot+[name path=F1,mark=none, color=black, domain=0.52381:1, thick, forget plot] {sqrt((1.1*1.1*(x-1)*(x-1)-x*x)/(1-1.1*1.1))};
						\addplot+[name path=G1,mark=none, color=black, domain=0.52381:1, thick, forget plot] {-sqrt((1.1*1.1*(x-1)*(x-1)-x*x)/(1-1.1*1.1))};
						\addplot+[name path=F2, mark=none, color=black, domain=0:0.47619, thick, forget plot] {sqrt((1.1*1.1*(-x)*(-x)-(1-x)*(1-x))/(1-1.1*1.1))};
						\addplot+[name path=G2, mark=none, color=black, domain=0:0.47619, thick, forget plot] {-sqrt((1.1*1.1*(-x)*(-x)-(1-x)*(1-x))/(1-1.1*1.1))};
						\addplot+[pattern=checkerboard, forget plot]fill between[of=F1 and G1, soft clip={domain=0.52381:1}];
						\addplot[pattern=checkerboard]fill between[of=F2 and G2, soft clip={domain=0:0.47619}];
						\addlegendentry{$R=1.1$}

						\addplot+[mark=none, color=violet,  domain=1:80, thick, forget plot] (1-1/(9.645908801*ln(100*x)),x);
						\addplot+[mark=none,color=violet,  domain=-1:1,thick, forget plot] (1-1/(9.645908801*ln(100)),x);
						\addplot+[mark=none,color=violet,  domain=-80:-1, thick, forget plot] (1-1/(9.645908801*ln(100*(-x))),x);
						\addplot+[mark=none,color=violet,  domain=1:80,thick, forget plot] (1/(9.645908801*ln(100*x)),x);
						\addplot+[mark=none,color=violet,  domain=-1:1,thick, forget plot] (1/(9.645908801*ln(100)),x);
						\addplot[mark=none,color=violet,  domain=-80:-1, thick] (1/(9.645908801*ln(100*(-x))),x);
						\addlegendentry{McCurley}

						\addplot+[mark=none, color=green, dashed, domain=-1:1, very thick, forget plot] (1-1/(5.6*ln(100)),x);
						\addplot+[mark=none, color=green, dashed, domain=1:80, very thick, forget plot] (1-1/(5.6*ln(100*x)),x);
						\addplot+[mark=none, color=green, dashed, domain=-80:-1, very thick, forget plot] (1-1/(5.6*ln(100*(-x))),x);
						\addplot[mark=none, color=green, dashed,  domain=-1:1, ultra thick] (1/(5.6*ln(100)),x);
						\addplot[mark=none, color=green, dashed,  domain=1:80, ultra thick] (1/(5.6*ln(100*x)),x);
						\addplot[mark=none, color=green, dashed,  domain=-80:-1, ultra thick] (1/(5.6*ln(100*(-x))),x);
						\addlegendentry{Kadiri}
		\end{axis}
\end{tikzpicture}
\caption{Different regions, $\left|\frac{\rho}{\rho-1}\right|\ge R$, McCurley \eqref{mccurleyRegion} and Kadiri \eqref{kadiriRegion}.}
\label{pic1}
\end{figure}
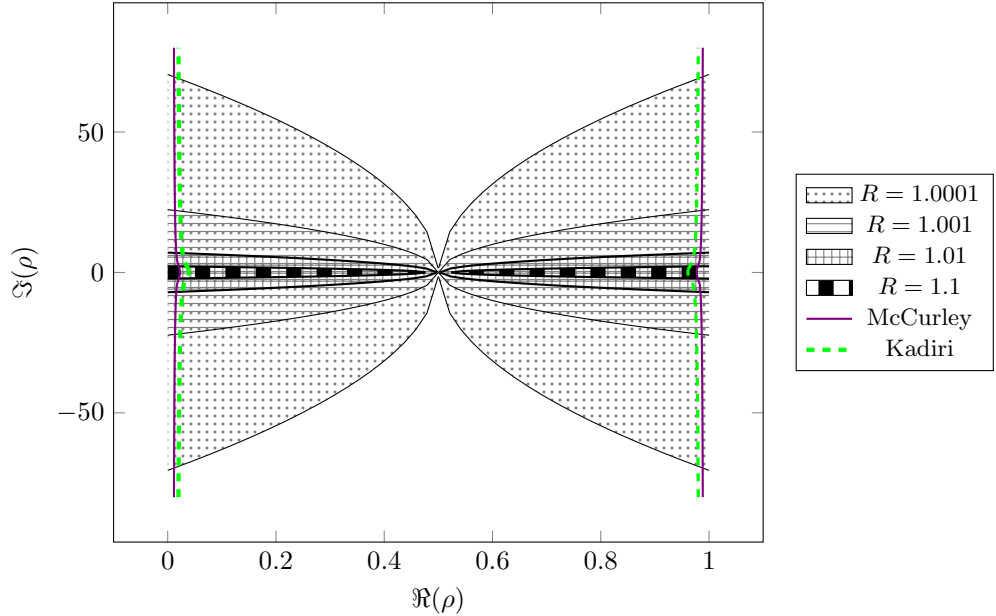

\subsection{Dirichlet \texorpdfstring{$L$}{}-functions and Theorems \ref{mainresult1} and \ref{resultExp1}}
\label{subsubsDiriFirst}
In this section we consider what Theorems \ref{mainresult1} and \ref{resultExp1} and Corollaries \ref{corolDiri1} and \ref{coroDiri2} state for Dirichlet $L$-functions associated with a primitive non-principal character modulo $q=100$. Recall that according to Theorem \ref{mainresult1} if certain $\tau$-Li coefficients are non-negative, then all of the zeros lie outside certain region. Furthermore, according to Theorem \ref{resultExp1}, if there exists exactly one zero $\rho$ with $\left|\frac{\rho}{\rho-\tau}\right|\ge R$, then $|\lambda_F(n,\tau)|$ is large enough for a certain integer $n$. The numerical results can be seen in Table \ref{Diri1}. 

\begin{table}
\caption{Different values of the terms defined in Theorem \ref{mainresult1}/Corollary \ref{corolDiri1} and in Theorem \ref{resultExp1}/Corollary \ref{coroDiri2}. Here $F$ is a Dirichlet $L$-function associated with a primitive non-principal character modulo $100$.}
\label{Diri1}
\begin{center}
\begin{tabular}{lllll}
\hline\noalign{\smallskip}
$\tau$ &  $R$ &Theorem \ref{mainresult1}:&Theorem \ref{mainresult1}: & Theorem \ref{resultExp1}: \\
  &   & $N$ & $5N^2\left(A_F\log{N}+M_F\right)$ &  $N$ \\
\midrule
1 & 1.1 & 2228 & \num{142795217} & \num{838} \\
1 &1.01 & \num{21335} & \num{14730585353} & \num{8027}  \\
1 & 1.001 & \num{212394} & \num{1624882482585} & \num{79909} \\
1 & 1.0001 & \num{2122983} & \num{178855533212062} & \num{798729} \\
1.5 &1.1 & \num{3551}  & \num{372085537} & \num{876}  \\
1.5 &1.01 & \num{34009} & \num{38288548586} & \num{8387}  \\
1.5 & 1.001 & \num{338570} & \num{4213969451870} & \num{83491} \\
1.5 & 1.0001 & \num{3384171} & \num{462978327657268}  & \num{834529} \\
\noalign{\smallskip}\hline\noalign{\smallskip}
\end{tabular}
\end{center}
\end{table}

\subsection{Dirichlet \texorpdfstring{$L$}{}-functions and proof of Theorem \ref{mainresult2}}
\label{subsubsDiriSecond}
Next we consider what the proof of Theorem \ref{mainresult2} states for the Dirichlet $L$-function associated with a primitive non-principal character modulo $q=100$. The goal is to find numbers $n$ such that if $\lambda_F(n,\tau)< 0$, then there exists at least one zero $\rho$ with $\left|\frac{\rho}{\rho-\tau}\right|\ge R$. The reason why we consider the proof instead of using the formulas proved in Theorem \ref{mainresult2} is that this way the results are a little bit sharper than in the general case proved in Theorem \ref{mainresult2}. We do not prove full formulas for this case. Instead, we compute numerical examples which can be seen in Table \ref{Diri2}. 

Recall $\tau \ge 1$ and $R>1$. In the proof of Theorem \ref{mainresult2} we proved the claim by showing that if for all the zeros $\rho$ it holds $\left|\frac{\rho}{\rho-\tau}\right|\ge R$, then certain $\tau$-Li coefficients are non-negative. The same method is used here. We also used a variable $T$ such that $T \ge \max\{T_0,ne\tau\}$. Using similar methods as in the proof of Theorem \ref{mainresult2}, we want 
\begin{align}
		\begin{split}
				& \frac{0.432\tau^2}{T}\left(\frac{1}{2\pi}\log{(2T)}+\frac{1}{2\pi}\log{\frac{200}{\pi e}} \right. \\
				& \quad \left.-\frac{0.634}{3T}\log(2^\frac{1}{3}T) -\frac{1}{3T}(0.317\log{2}+0.634\log{100}+12.802)\right ) \\
				& \quad  \coloneqq K_{F,3}(T,\tau) \\
				& \quad >0.
		\end{split}
\end{align}
This holds for $T\ge 6.348$. Further, by the proof of Theorem \ref{mainresult2} and previous estimates, we have 
\begin{equation*}
\begin{aligned}
			\lambda_F(n,\tau ) &>K_{F,3}(T,\tau)n(n-1)+\mathcal{N}_F(T)(1-R^n).
\end{aligned}			
\end{equation*}
We want the right-hand side of the previous inequality to be at least zero for $n \in [n_0,n_1]$, where $n_0$ and $n_1$ are positive integers and $n_1\le \frac{T}{e\tau}$. 

We have computed different values for the terms $n_0$ and $n_1$ using different values for the terms $T$, $\tau$ and $R$. In the computations we have used formula \eqref{diricformula} for the term $\mathcal{N}_F(T)$, recalling that the number of zeros is always an integer and we always have $\Re(\lambda_F(1,\tau))\ge 0$. The results can be seen in Table \ref{Diri2}. 

We can notice some results from the table. For example, if the coefficient $\Re\left(\lambda_F(n,1)\right)$ is negative for some integer $n \in [5,36]\cup[95,183]$, then there exists at least one zero $\rho$ with $\left|\frac{\rho}{\rho-1}\right|\ge 1.0001$.

\begin{table}
\caption{Different values for the terms $n_0$ and $n_1$ such that if for some $n \in [n_0,n_1]$ it holds $\Re\left(\lambda_F(n,\tau)\right)< 0$, then there exists at least one zero $\rho$ with $\left|\frac{\rho}{\rho-\tau}\right|\ge R$. Here $F$ is a Dirichlet $L$-function associated with a primitive non-principal character modulo $100$.}
\label{Diri2}
\begin{center}
\begin{tabular}{lllll}
\hline\noalign{\smallskip}
$T$ & $\tau$ & $R$ & $n_0$& $n_1$   \\
\noalign{\smallskip}\hline\noalign{\smallskip}
100 & 1 & 1.0001 & 5 & 36 \\
100 & 1 & 1.00001 & 1 & 36 \\
500 & 1 & 1.0001 & 95 & 183 \\
500 & 1 & 1.00001 & 11 & 183  \\
500 & 1 & 1.000001 & 1 & 183  \\
\num{10 000} & 1 & 1.000001 & 391  & 3678  \\
\num{10 000} & 1 & 1.0000001 & 40  & 3678  \\
\num{10 000} & 1 & 1.00000001 & 5  & 3678  \\
\num{10 000} & 1 & 1.000000001 & 1  & 3678  \\
100 & 1.5 & 1.001 & 19  & 24  \\
100 & 1.5 & 1.0001 & 3  & 24  \\
100 & 1.5 & 1.00001 &  1 & 24  \\
500 & 1.5 & 1.0001 & 43  & 122  \\
500 & 1.5 & 1.00001 & 6  & 122  \\
500 & 1.5 & 1.000001 & 1  & 122  \\
\num{10 000} & 1.5 & 1.00001 & 1748  & 2452  \\
\num{10 000} & 1.5 & 1.000001 & 175  & 2452  \\
\num{10 000} & 1.5 & 1.0000001 & 19  & 2452  \\
\num{10 000} & 1.5 & 1.00000001 & 3  & 2452  \\
\num{10 000} & 1.5 & 1.000000001 & 1  & 2452  \\
\noalign{\smallskip}\hline
\end{tabular}
\end{center}
\end{table}

\section{Example: \texorpdfstring{$L$}{}-function associated with a holomorphic newform with a level \texorpdfstring{$1$}{} and a weight \texorpdfstring{$12$}{}}
\label{subsL}
In this section we assume that the function $F(s)$ is an $L$-function associated with a holomorphic newform with level $1$ and weight $12$. We also assume that all zeros $\rho$ with $|\Im(\rho)|\le 27$ lie on the critical line. The function $F(s)$ is in the Selberg class and thus satisfies conditions \ref{locationConditionMore}-\ref{computationCondition}. We consider what Theorem \ref{resultExp1} states for the function $F(s)$. Theorem $\ref{resultExp1}$ says that if the function $F(s)$ has exactly one zero $\rho$ with $\left|\frac{\rho}{\rho-\tau}\right|\ge R$, then $|\lambda_F(n,\tau)|$ is large enough for an integer $n$ which lies in a certain interval. 

By G. Fran{\c c}a and A. LeClair \cite[Table IX]{franca} , $\mathcal{N}_F(27)=14$. Thus and by \cite[Table 1]{palojarvi} we have
\begin{equation*}
		\begin{aligned}
				& A_F=\frac{1}{\pi}, \quad B_F=-\frac{1}{\pi}(1+\log{(4\pi^2)}), \quad T_0=27, \\
				& C_{F,1}(T_0)=586, \quad C_{F,2}(T_0)=3904, \quad C_{F,3}(T_0)=\num{23274}, \\
				& c_{F,1}(T_0)=864, \quad c_{F,2}(T_0)=3622 \quad \text{and} \quad c_{F,3}(T_0)=\num{21012}.
		\end{aligned}
\end{equation*}
Using these constants we can formulate the following corollary which describes Theorem \ref{resultExp1} for the function $F(s)$:
\newpage
\begin{corollary}
\label{corolHolo}
Let $F$ be a $L$-function associated with a holomorphic newform with level $1$ and weight $12$ and let $R>1$ and $\tau \ge 1$ be real numbers. Then the number $N$ defined in Theorem \ref{resultExp1} can be written as
\[
\begin{aligned}
N &= \ceil[\Bigg]{\max\left\{\exp\left(-W_{-1}\left(-\frac{2\log{R}}{3}\right)\right), \frac{12\log{\left(40\left(0.5+K_{F,1}(\tau)+K_{F,4}(\tau)\right)\right)}}{\log{R}}\right\}}
\end{aligned}
\]
if $R \le e^{\frac{3}{2e}}$, and
\[
N = \ceil[\Bigg]{\max\left\{e, \frac{27}{e\tau}, \frac{12\log{\left(40\left(0.5+K_{F,1}(\tau)+K_{F,4}(\tau)\right)\right)}}{\log{R}}\right\}}
\]
otherwise.

Here 
\begin{align*}
				K_{F,1}(\tau)&= \frac{2\tau}{3\pi}\left(e+\frac{1}{e}\right)\left(1+\left|\log{\frac{2\tau}{\pi^2}}\right|\right) \\
				& \quad +\frac{4}{27}\left(1+\frac{1}{e^2}\right)\left(288\log{(2e^6\tau^3)}+3622+\frac{\num{21012}}{7e\tau}\right)
\end{align*}
and
\begin{equation*}
\begin{aligned}
K_{F,4}(\tau)&= 2\left(\vphantom{\frac{C_{F,3}(T_0)}{e^3\tau}} \frac{1}{\pi}e\tau\log{(4e^3\pi^2\tau)}+\frac{586\log{(e^2\tau)}+3904}{3}+\frac{2586}{e\tau}\right).
\end{aligned}
\end{equation*}
\end{corollary}
Please notice that the term $N$ in the previous corollary cannot be simplified since the elements in the sets depend on different sets of variables.

We also compute the values for the number $N$ described in Corollary \ref{corolHolo}. The results can be seen in Table \ref{holomorphicMresults}.

\begin{table}
\caption{Different values for the term $N$ which is defined in Theorem \ref{resultExp1}/Corollary \ref{corolHolo}. Here $F$ is a $L$-function associated with a holomorphic newform with a level $1$ and a weight $12$.}
\label{holomorphicMresults}
\begin{center}
\begin{tabular}{lll}
\hline\noalign{\smallskip}
$\tau$ &  $R$ & Theorem \ref{resultExp1}/Corollary \ref{corolHolo}: $N$ \\
\noalign{\smallskip}\hline\noalign{\smallskip}
1 & 1.0001 & \num{1498217} \\
1 & 1.00001 & \num{14981490} \\
1 & $1+10^{-10}$ & \num{1498141425042} \\
1.5 & 1.0001 & \num{1488111} \\
1.5 & 1.00001 & \num{14880440} \\
1.5 &  $1+10^{-10}$ & \num{1488036546102} \\
\noalign{\smallskip}\hline
\end{tabular}
\end{center}
\end{table}

\subsection*{Acknowledgements}
I would like to thank professor Lejla Smajlovi\'{c} for helpful comments for the manuscript.

\nocite{*}
\bibliographystyle{plain}
\bibliography{article_tau_li}{}

\end{document}